\documentclass[11pt,twoside]{article}

\usepackage{mathrsfs,amsfonts,amsmath,amssymb}
\usepackage{latexsym}
\newtheorem{satz}{Theorem}
\newtheorem{proposition}[satz]{Proposition}
\newtheorem{theorem}[satz]{Theorem}
\newtheorem{lemma}[satz]{Lemma}
\newtheorem{definition}[satz]{Definition}
\newtheorem{corollary}[satz]{Corollary}
\newtheorem{remark}[satz]{Remark}

\newcommand{\lf}{\left\lfloor}
\newcommand{\rf}{\right\rfloor}

\def\no{\noindent}
\def\sbeq{\subseteq}
\def\N{\mathbb {N}}
\def\Z{\mathbb {Z}}
\def\E{\mathsf {E}}
\def\T{\mathsf {T}}
\def\R{\mathbb {R}}
\def\F{\mathbb {F}}
\def\M{\mathbf {M}}
\def\e{\varepsilon}
\def\l{\lambda}

\def\I{{\cal I}}
\def\a{\alpha}
\def\m{\times}

\def\P{{\cal P}}
\def\h{\widehat}
\def\G{\Gamma}
\def\a{\alpha}

\def\g{\gamma}
\def\d{\delta}
\def\D{\Delta}
\def\o{\overset}
\def\({\big (}
\def\){\big )}
\def\g{\gamma}
\def\c{\circ}
\def\k{\kappa}
\def\L{\Lambda}
\def\l{\lambda}
\def\b{\beta}
\def\t{\tau}
\def\ls{\leqslant}
\def\gs{\geqslant}

\def\_phi{\varphi}
\def\eps{\varepsilon}
\def\m{\times}
\def\Gr{{\mathbf G}}
\def\FF{\widehat}
\def\ov{\overline}

\def\oT{{\rm T}}

\def\supp{{\rm supp\,}}

\def\ge{\gs}
\def\le{\ls}

\begin{document}

\title{\bf Higher moments of convolutions}

\author{ By\\  \\{\sc Tomasz Schoen\footnote{The author is supported by MNSW grant N N201 543538.} ~ and ~
Ilya D. Shkredov\footnote{This work was supported by grant RFFI NN
06-01-00383, 11-01-00759, Russian Government project 11.G34.31.0053 and grant Leading Scientific Schools N
8684.2010.1.}} }
\date{}

\maketitle

\begin{center}
    Abstract
\end{center}

{\it \small
    We study higher moments of convolutions of the characteristic function of a set,
    which generalize a classical notion of the additive energy.
    Such quantities appear in many problems of additive combinatorics as well as in number theory.
    In our investigation we use different approaches including basic combinatorics, Fourier analysis and eigenvalues method to establish
    basic properties of higher energies. We provide also a sequence of applications of higher energies additive combinatorics.
}
\\

\section{Introduction}

Let $\Gr$ be an abelian group, and $A\subseteq \Gr$ be an arbitrary finite set.
The {\it additive energy} of the set $A$ is defined by
$$
    \E_2 (A) = |\{ a_1-a_2 = a_3-a_4 ~:~ a_1,a_2,a_3,a_4 \in A \}| \,.
$$
This quantity plays an important role in many problems of additive combinatorics as well as in number theory (see e.g. \cite{tv}).
In the article we study, basically, the following generalization of the additive energy
$$
    \E_k (A) = |\{ a_1-a_2 = a_3-a_4 = \dots = a_{2k-1}-a_{2k} ~:~ a_1,\dots,a_{2k} \in A \}| \,, \quad k\ge 2 \,.
$$
Geometrically, $\E_k(A)$ is the number of $k$--tuples of Cartesian product $A^k$,
which belong to the same line from the system of lines of the form $y=x+c$, $c\in A-A$.
An analog of $\E_3(A)$ for general systems of lines and points has applications
in combinatorial geometry and  in sum--product problems (see \cite{tv}, chapter 8).
$\E_k(A)$ can be also expressed as the $k$th moment of the  convolution of the characteristic function of  the set $A$
(see \cite{sv}).

Higher energies have already found some applications (see
\cite{ss,ss-convex,sv}). Here we collect further properties and
applications of $\E_k(A)$.
To prove them we  use different approaches including basic combinatorics, Fourier analysis and eigenvalues method.

The paper is organized as follows.
We start  with definitions and notations used in the paper.
In the next section we consider some basic properties of higher energies.
We prove, in particular, that the smallness of energy $\E_k$
implies a non--trivial upper bound for the cardinality of the set of large Fourier coefficients and vice versa.
These sets play an important role in every problem of additive combinatorics where Fourier analysis is used (see \cite{tv}).

Quantities $\E_k(A)$ can be expressed in terms of higher convolutions of the set $A$
(see \cite{sv}).
We continue to study supports of these convolutions in  section  \ref{sec:Ruzsa_triangle}.
We establish a  generalization of  Rusza's triangle inequality, which allows us to
introduce
 a hierarchy of bases of abelian groups (that is sets $B$ with $B\pm B = \Gr$) and prove some its properties.

In  section \ref{sec:CS} we show that the knowledge of energies
$\E_k$ allows to refine Croot--Sisask almost periodicity lemma (see
\cite{cs}). Further, we prove in section
\ref{sec:E_3_energy_&_structure}, that for any $A$ with $|A-A| =
K|A|$ and $\E_3(A) = M|A|^4/K^2$, for a relatively small $M,$ there
is a large subset $A' \subseteq A$ such that $A'$ has almost no
growth under addition. To show this result, we use a technique
introduced in \cite{ss}, where among other things Katz--Koester
transform \cite{kk} is applied.
Series of results contained in  sections
\ref{sec:CS}--\ref{sec:BSzG} can be considered as statements on
structure of sets with small $\E_3 (A)$ (results on structure of
sets with small proportion of two another generalizations of the
additive energy can be found in \cite{bk1,bk2}).

In  section \ref{sec:product_E_k} we prove some results related
to sum--product problem in $\R$. Solymosi \cite{s} showed an upper
bound for multiplicative  energy in terms of the size of the sumset
$A+A.$ Improving a theorem of Li  \cite{Li}, we prove an upper
estimate of $\E_k(A)$ in terms of $|A\cdot A|.$ Our approach is
based on Szemer\'edi--Trotter theorem and develops some ideas
introduced in \cite{ss-convex} and \cite{Li}.



In the
next section we use so--called eigenvalue method to study
$\E_k (A)$. Using this approach, we show that the magnification
ratio of a set $A$
(see \cite{tv} and also \cite{p}) is closely related with the
behavior of $\E_k(A)$. Actually, it turns out that the method allows
to prove lower bounds for the cardinality of restricted sumsets
$A\overset{G}{+}B$, where $G$ is a subgraph of the complete
bipartite graph with bipartition $A,B$
 (see Theorem \ref{t:R&eighenvalues}). As an application,  we
obtain some results concerning  sumsets of sets with small $\E_k$
(another application will be also given in the section
\ref{sec:BSzG}). The results are particularly powerful in the case
of multiplicative subgroups of the field $\F_q$.

In  section \ref{sec:BSzG} we prove two versions of the
well--known Balog--Szemer\'{e}di--Gowers
\cite{balog,balog-szemeredi,gowers}.
Assuming $\E_2(A)=|A|^3/K$  and $\E_3(A)\le M|A|^4/K^2$ we obtain an improvement of
Balog--Szemer\'{e}di--Gowers theorem, and with the assumptions $\E_2(A)=|A|^3/K$  and  $\E_4(A)\le M|A|^{5}/K^{3}$ we show an optimal
version of Balog--Szemer\'{e}di--Gowers theorem .

Finally, in the last section we  prove some
results, which connects higher energies and higher
moments of the Fourier transform of $A$.

I.D.S. is grateful to A.V. Akopyan and F. Petrov for useful discussions. Both
authors are grateful to N.G. Moshchevitin and  V.F. Lev.
I.D.S.
thanks
Institute IITP RAN for excellent working conditions.

\section{Notation}
\label{sec:notation}

Let $\Gr$ be an abelian group.
If $\Gr$ is finite then denote by $N$ the cardinality of $\Gr$.
It is well--known~\cite{Rudin_book} that the dual group $\FF{\Gr}$ is isomorphic to $\Gr$ in the case.
Let $f$ be a function from $\Gr$ to $\mathbb{C}.$  We denote the Fourier transform of $f$ by~$\FF{f},$
\begin{equation}\label{F:Fourier}
  \FF{f}(\xi) =  \sum_{x \in \Gr} f(x) e( -\xi \cdot x) \,,
\end{equation}
where $e(x) = e^{2\pi i x}$.
We rely on the following basic identities
\begin{equation}\label{F_Par}
    \sum_{x\in \Gr} |f(x)|^2
        =
            \frac{1}{N} \sum_{\xi \in \FF{\Gr}} \big|\widehat{f} (\xi)\big|^2 \,.
\end{equation}
\begin{equation}\label{svertka}
    \sum_{y\in \Gr} \Big|\sum_{x\in \Gr} f(x) g(y-x) \Big|^2
        = \frac{1}{N} \sum_{\xi \in \FF{\Gr}} \big|\widehat{f} (\xi)\big|^2 \big|\widehat{g} (\xi)\big|^2 \,.
\end{equation}
If
$$
    (f*g) (x) := \sum_{y\in \Gr} f(y) g(x-y) \quad \mbox{ and } \quad (f\circ g) (x) := \sum_{y\in \Gr} f(y) g(y+x)
$$
 then
\begin{equation}\label{f:F_svertka}
    \FF{f*g} = \FF{f} \FF{g} \quad \mbox{ and } \quad \FF{f \circ g} = \ov{\FF{\ov{f}}} \FF{g} \,.
\end{equation}
For a function $f:\Gr \to \mathbb{C}$ put $f^c (x):= f(-x)$.
 Clearly,  $(f*g) (x) = (g*f) (x)$, $x\in \Gr$.
 The $k$-fold convolution, $k\in \N$  we denote by $*_k$,
 so $*_k := *(*_{k-1})$.

Write $\E(A,B)$ for {\it additive energy} of two sets $A,B \subseteq \Gr$
(see e.g. \cite{tv}), that is
$$
    \E(A,B) = |\{ a_1+b_1 = a_2+b_2 ~:~ a_1,a_2 \in A,\, b_1,b_2 \in B \}| \,.
$$
We use in the paper  the same letter to denote a set
$S\subseteq \Gr$ and its characteristic function $S:\Gr\rightarrow \{0,1\}.$

If $A=B$ we simply write $\E(A)$ instead of $\E(A,A).$ Clearly,
\begin{equation}\label{f:energy_convolution}
    \E(A,B) = \sum_x (A*B) (x)^2 = \sum_x (A \circ B) (x)^2 = \sum_x (A \circ A) (x) (B \circ B) (x)
    \,,
\end{equation}
and by (\ref{svertka}),
\begin{equation}\label{f:energy_Fourier}
    \E(A,B) = \frac{1}{N} \sum_{\xi} |\FF{A} (\xi)|^2 |\FF{B} (\xi)|^2 \,.
\end{equation}

Let
$$
   \T_k (A) := | \{ a_1 + \dots + a_k = a'_1 + \dots + a'_k  ~:~ a_1, \dots, a_k, a'_1,\dots,a'_k \in A \} | \,.
$$
Generally,  for every function $f: \Gr \to \mathbb{C}$ set
$\T_k (f) = \sum_x |(f *_{k-1} f) (x)|^2$.
Clearly, $\T_k (A) = \frac{1}{N}  \sum_{\xi} |\FF{A} (\xi)|^{2k}$.
Let also
$$
    \sigma_k (A) := (A*_k A)(0)=| \{ a_1 + \dots + a_k = 0 ~:~ a_1, \dots, a_k \in A \} | \,.
$$
Notice that for a symmetric set $A$ that is $A=-A$ one has $\sigma_2
(A) = |A|$ and $\sigma_{2k} (A) = \T_k (A)$.

 For a sequence $s=(s_1,\dots, s_{k-1})$ put
$A_s=A\cap (A-s_1)\dots \cap (A-s_{k-1}).$ Let
\begin{equation}\label{f:E_k(A)_def}
    \E_k(A)=\sum_{x\in \Gr} (A\c A)(x)^k = \sum_{s_1,\dots,s_{k-1} \in \Gr} |A_s|^2
\end{equation}
and
\begin{equation}\label{f:E_k(A,B)_def}
    \E_k(A,B)=\sum_{x\in \Gr} (A\c A)(x) (B\c B)(x)^{k-1} \,.
\end{equation}
Similarly, we write $\E_k(f,g)$ for any complex functions $f$ and $g$.
Putting $\E_1 (A) = |A|^2$.

We shall write $\sum_x$ and $\sum_{\xi}$ instead of $\sum_{x\in \Gr}$ and $\sum_{\xi \in \FF{\Gr}}$ for simplicity.

For a positive integer $n,$ we set $[n]=\{1,\ldots,n\}$.
All logarithms used in the paper are to base $2.$
By  $\ll$ and $\gg$ we denote the usual Vinogradov's symbols.

\section{Basic properties of higher  energies}
\label{sec:Ruzsa_triangle}

Here we collect basic properties of $\E_k(A)$, where $A$ is a finite subset of an abelian group $\Gr.$ If
$|A-A|=K|A|$ then
$$\E_k(A)\ge \frac{|A|^{k+1}}{K^{k-1}}.$$
The first very useful property of higher energy was proved in \cite{ss} and \cite{sv}. The next lemma is a special case
of  Lemma 2.8 from \cite{sv}.

\begin{lemma}\label{l:E_k-identity} Let $A$ be a subset of an abelian group.
Then for every $k,l\in \N$
$$\sum_{s,t:\atop \|s\|=k-1,\, \|t\|=l-1} \E(A_s,A_t)=\E_{k+l}(A) \,,$$
where $\|x\|$ denote the number of components of vector $x$.
\end{lemma}

\begin{lemma}\label{l:E_k-identity'} Let $A$ be a subset of an abelian group.
Then for every $\a\in \R$
$$\sum_{\|s\|=1} \E_{1+\a}(A_s,A)=\E_{2+\a}(A) \,.$$
\end{lemma}

\begin{lemma}\label{l:sigma_k}
Let $A$ be a subset of an abelian group.
Then for every $k\in \N$, we have
\begin{equation}\label{f:E_k_&_sigma_k}
    |A|^{2k} \le \E_k (A) \cdot \sigma_k (A-A) \,, \quad \quad |A|^{4k} \le \E_{2k} (A) \cdot \T_{k} (A+A) \,,
\end{equation}
and
\begin{equation}\label{f:E_k_&_E_k}
    |A|^{2k+4} \le \E_{k+2} (A) \cdot \E_k (A-A) \,, \quad \quad |A|^{2k+4} \le \E_{k+2} (A) \cdot \E_k (A+A) \,.
\end{equation}
\end{lemma}
\begin{proof}
    Let us prove the first inequality from (\ref{f:E_k_&_sigma_k}).
    The formula is trivial for $k=1$, so
    suppose that $k\ge 2$.
    Consider the map
    $$\_phi : A^{k} \to (A-A)^{k}$$
    defined by
$$
    \_phi (a_1,\dots, a_{k}) = (a_1-a_2, a_2-a_3, \dots, a_{k-1}-a_{k}, a_{k}-a_1) = (x_1,\dots,x_{k}) \,.
$$
Clearly, $x_1+\dots+x_{k} = 0$.
Thus $\sigma_{k} (A-A) \ge |\mathbf{Im} (\_phi)|$.
By Cauchy--Schwarz inequality
\begin{eqnarray*}
    |A|^{2k} &\le& |\mathbf{Im} (\_phi) | \cdot |\{ z,w\in A^k ~:~ \_phi(z) = \_phi(w)
    \}|\\
        &\le&
            \sigma_{k} (A-A) \cdot |\{ z,w\in A^k ~:~ \_phi(z) = \_phi(w) \}| \,.
\end{eqnarray*}
To finish the proof it is enough to observe that
$$
    |\{ z,w\in A^k ~:~ \_phi(z) = \_phi(w) \}| = \E_k (A) \,.
$$
To obtain the second inequality from (\ref{f:E_k_&_sigma_k}), consider
$$
    \_phi' (a_1,\dots, a_{2k}) = (a_1+a_2, a_2+a_3, \dots, a_{2k-1}+a_{2k}, a_{2k}+a_1) = (x_1,\dots,x_{2k}) \,.
$$
instead of $\_phi$.
Because of $x_1-x_2+x_3-x_4+ \dots + x_{2k-1}-x_{2k} = 0$
and
$$
    |\{ z,w\in A^k ~:~ \_phi'(z) = \_phi'(w) \}| = \E_k (A) \,.
$$
we can use the previous arguments.

To obtain the first inequality in (\ref{f:E_k_&_E_k}) consider the map
$$\psi : A^{k+2} \to (A-A)^{2k}$$ defined by
$$
    \psi (b_1,b_2,a_1,\dots, a_{k}) = ( b_1-a_1, b_2-a_1, \dots, b_1-a_k, b_2-a_k ) = (x_1,y_1,\dots,x_k,y_k) 
$$
and similar with pluses.
It is easy to check that
$$
    |\{ z,w\in A^{k+2} ~:~ \psi(z) = \psi(w) \}| = \E_{k+2} (A)
$$
and $\E_{k} (A-A) \ge |\mathbf{Im} (\psi)|$ because of
$$
    x_1-y_1 = \dots = x_k - y_k \,.
$$
Thus, we obtain (\ref{f:E_k_&_E_k}) by the arguments above.
$\hfill\Box$
\end{proof}
\bigskip

 It turns out that $\E_k(A)$ is also closely related with higher
dimensional sumsets. Observe that
\begin{eqnarray}\label{f:energy-B^k-Delta}
\E_{k+1}(A,
B)&=&\sum_x(A\c A)(x)(B\c B)(x)^{k}\nonumber \\
&=&\sum_{x_1,\dots, x_{k-1}}\Big (\sum_y A(y)B(y+x_1)\dots
B(y+x_{k})\Big )^2 =\E(\Delta(A),B^{k})
 \end{eqnarray}
and
\begin{eqnarray*}
\sum_{x\in X} (A\c B)(x)^k &=&\sum_{x\in X}|\{(a_1,b_1), \dots, (a_k,b_k)\in A\times B: b_1-a_1=\dots=b_k-a_k=x\}|\\
&=& \sum_{y\in A^k}(\D(X)\c B^k) (y)\,,
\end{eqnarray*}
where
$$
    \Delta (A) = \Delta_k (A) := \{ (a,a, \dots, a)\in A^k \}\,.
$$
We also put $\Delta(x) = \Delta (\{ x \})$, $x\in \Gr$.
The formula above gives a motivation  to study the sumsets $A^k-\D(A),$
where
$A^k, \D(A)\sbeq \Gr^k.$
Another motivation to study such sets was discussed in \cite{sv}.
It turns out that these sets appear naturally as
supports
of higher convolutions of the set $A$.

Clearly
$$A^k-\D(A)=\bigcup_{a\in A} (A-a)^k \text{~~ and ~~ } A^k+\D(A)=\bigcup_{a\in A} (A+a)^k\,. $$
By Cauchy-Schwarz inequality we have
\begin{equation}\label{f:higher_sumsets_and_E_k+1}
    |A^k-\D(A)|\ge \frac{|A|^{2k+2}}{\E(A^k,\D(A))}=\frac{|A|^{2k+2}}{\E_{k+1}(A)} \,.
\end{equation}
Trivially for every $A_1,\dots A_k\sbeq \Gr$
\begin{equation}\label{f:Ruzsa_triangle'}
    |A_1  \m \dots \m A_{k-1} - \Delta(A_k)| \ls \min \Big (\prod_{i=1}^k |A_i|\, ,\prod_{j=1}^{k-1} |A_j - A_k|\Big ) \,.
\end{equation}


\bigskip

Now assume that  $\Gr$ is  a finite abelian group and $A\subseteq \Gr$.
For any $\a\in (0,1]$ put
$$R_{\a} = R_{\a} (A)=\{r \in \FF{\Gr} ~:~ |\h A (r)|\gs \a |A| \}\,.$$
Thus, $R_{\a} (A)$ is the set of large Fourier coefficients of the set $A$.
We show that the size and the structure of $R_\a$ is highly related to $\E_k(A).$
We make use of the following lemma, which  was proved in \cite{shkredov_LES,shkredov_dissociated}.

\begin{lemma}\label{lemma:Shkredov}
Let $\a \in (0,1]$ be a real number.
Let also $A$ be a subset of a finite abelian group $\Gr$, $|A|=\d N$, and
let $\L\sbeq R_\a\setminus \{0\}.$ Then
$$\T_k(\L)\ge \d\a^{2k}|\L|^{2k} \,.$$
\end{lemma}

\begin{theorem}
Let $\a \in (0,1]$ be a real number. Suppose that $A$ is a subset of
an abelian group $\Gr$ of order $N$ and $|A|=\d N.$ Suppose that
$\E_{k}(A)=\k_{k}|A|^{k+1}.$ Then
\begin{equation}\label{f:R_a_E_k}
    |R_{\a}|\ls \a^{-3}\d^{-1}(\k_{2k}-\d^{2k-1})^{1/2k} \,,
\end{equation}
and
\begin{equation}\label{f:R_a_E_k'}
    \max_{r\not=0}|\h A(r)|\gs k^{-1/2}(\k_k-\d^{k-1})^{1/2}|A|\,.
\end{equation}
Moreover, $\kappa_k \ge \kappa_{k-1}^{\frac{k-1}{k-2}}\ge \d
\kappa_{k-1}$,
and
\begin{equation}\label{f:R_a_E_k''}
    \max_{r\not=0}|\h A(r)|\gs (\kappa_k - \d \kappa_{k-1})^{1/2}|A|\,.
\end{equation}
\end{theorem}
\begin{proof} By Fourier inversion formula
\begin{equation}\label{tmp:13.09.2011_1}
    \E_{2k}(A)=\sum_{t}(A\circ A)(t)^{2k}=
    \sum_t\Big (N^{-1}\sum_{r} |\h A (r)|^2e(tr)\Big )^{2k}=N^{1-2k}\sum_{\sum r_i=0}|\h A (r_1)|^2\dots |\h A (r_{2k})|^2 \,.
\end{equation}
Lemma \ref{lemma:Shkredov} implies that
$$\k_{2k}|A|^{2k+1}\gs \d^{2k-1}|A|^{2k+1}+N^{1-2k}\d \a^{2k}|R_\a|^{2k}(\a|A|)^{4k},$$
which gives the first inequality.

Next, notice that
$$\k_k|A|^{k+1}\ls \d^{k-1}|A|^{k+1}+k\max_{r\not=0}|\h A(r)|^2N^{1-k}(\sum_r |\h
A(r)|^2)^{k-1}=\d^{k-1}|A|^{k+1}+k\max_{r\not=0}|\h A(r)|^2
|A|^{k-1}\,,$$
 and we have proved (\ref{f:R_a_E_k'}).

 Finally, let us show
 (\ref{f:R_a_E_k''}).
H\"older inequality gives  $\kappa_k \ge
\kappa_{k-1}^{\frac{k-1}{k-2}},$ so that $\kappa_k\ge \d
\kappa_{k-1}$. For $k\gs 2$ put $\_phi (x) = (A\circ A)^{k-1} (x)$.
 Again, by the inverse formula
 $$
    \E_k (A) = \k_k |A|^{k+1} = \frac{1}{N} \sum_r |\h {A} (x)|^2 \h{\_phi} (x)
        \ls
            \k_{k-1} \d |A|^{k+1} + \max_{r\not=0}|\h A(r)|^2 |A|^{k-1}
 $$
 and the assertion follows.$\hfill\Box$
\end{proof}
\bigskip

\noindent Clearly, the inequality (\ref{f:R_a_E_k}) is better than trivial bound $|R_\a| \le \a^{-2}\d^{-1},$
provided that $$\a> (\k_{2k}-\d^{{2k-1}})^{1/2k}\,.$$

Next, we show that $A\pm A$ contains long arithmetic progressions and even more general configurations.
The first part of the proof of the corollary below uses an idea of Vsevolod Lev the second part is rather
similar to the method introduced in \cite{crs}.

\begin{corollary}
    Let $A \subseteq \Gr$ be a set, $|A| = \d N$.
    Let also $k \gg \log N/ \log (1/\d)$ and $c_1,\dots,c_k$ are any numbers not all equals zero.
    Then $A\pm A$ contains a configuration of the form $x+c_1 d, \dots, x+c_k d$ with $d\neq 0$.
\end{corollary}
\begin{proof}
    We find
    a tuple $x+c_1 d, \dots, x+c_k d$
    in $A-A$ because the case $A+A$
    follows from the additional observation that there is $s\in \Gr$ such that $|A \cap (s-A)| \ge \d^2 N$
    and $A\cap (s-A) - A\cap (s-A) \subseteq A+A - s$.
    Let $\vec{1} = (1,\dots,1)$, $\vec{c} = (c_1,\dots,c_k)$, and $\vec{u} = (u_1,\dots,u_k)$.
    Assume the contrary and apply analog of formula (\ref{tmp:13.09.2011_1}), we get
\begin{eqnarray*}
    |A|^{k+1} &\ge& \E_k (A) = \sum_{x,d} (A\circ A) (x+c_1 d) \dots (A\circ A) (x+c_k d)\\
        &=&
            \frac{1}{N^{k-2}} \sum_{\langle \vec{u}, \vec{1} \rangle = \langle \vec{u}, \vec{c} \rangle = 0}
                        |\FF{A}(u_1)|^2 \dots |\FF{A}(u_k)|^2
                            \ge
                                \d^{2k} N^{k+2}
\end{eqnarray*}
and the result follows.

Now we give a non--abelian variant of the proof in the case $A-A$.
Suppose that $|A^k| > N^{k-1}$.
Then the sets $A^k + (d c_1,\dots, d c_k)$, $d\in \Gr$ cannot be disjoint.
It means that for some different $d',d''$ we have
$(A^k + (d' c_1,\dots, d' c_k)) \cap (A^k + (d'' c_1,\dots, d'' c_k)) \neq \emptyset$.
In other words $((d'-d'') c_1,\dots, (d'-d'') c_k ) \in (A-A)^k$.
Thus,
$|A^k|
\le N^{k-1}$ and the result follows.
$\hfill\Box$

\end{proof}

\section{Ruzsa's triangle inequality and bases of higher depth}

Next results provide basic relations between sizes of higher
dimensional sumsets. The following theorem generalizes the well--known  Ruzsa's
triangle inequality \cite{ruzsa2}.

\begin{theorem}
Let $k \gs 1$ be a positive integer, and
let $A_1,\dots,A_k,B$ be finite subsets of an abelian group $\Gr$.
Further, let $W,Y \subseteq \Gr^k$, and $X,Z \subseteq \Gr$.
Then
\begin{equation}\label{f:Ruzsa_triangle1}
    |W\m X| |Y-\Delta(Z)| \le |Y\m W \m Z - \Delta(X)| \,,
\end{equation}
\begin{equation}\label{f:Ruzsa_triangle2}
    |A_1 \m  \dots \m A_k - \Delta(B)|
        \le
            |A_1  \m \dots \m A_{m} - \Delta(A_{m+1})| |A_{m+1} \m \dots \m A_{k} - \Delta(B)|
\end{equation}
for any $m\in [k]$.
Furthermore, we have
\begin{equation}\label{f:Ruzsa_triangle''}
    |Y \m Z - \Delta(X)| = |Y\m X - \Delta(Z)| \,.
\end{equation}
\label{t:Ruzsa_triangle}
\end{theorem}
\begin{proof}
To show the first
inequality we apply Ruzsa's argument.
For every ${\bf a}\in Y - \Delta(Z)$
choose the smallest element (in any linear order of $Z$) $z \in Z$
such that ${\bf a}=(y_1-z,\dots,y_k-z)$
for some $(y_1,\dots,y_k) \in Y$.
Next, observe that the function
$$({\bf a},{\bf w},x)\mapsto (y_1-x,\dots,y_k-x,z-x,w_1-x,\dots,w_k-x) \,,$$
where ${\bf w} = (w_1,\dots,w_k) \in W$
from $(Y - \Delta(Z)) \times W \times X$ to $Y\m W \m Z - \Delta(X)$
is injective.

    To obtain the second inequality consider the following matrix
\begin{displaymath}
    \M =
        \left( \begin{array}{cccccc}
                    1 & 0 & \ldots & 0 & 0 & -1 \\
                    0 & 1 & 0 & \ldots & 0 & -1 \\
                    0 & 0 & 1 & \ldots & 0 & -1 \\
                    \ldots & \ldots & \ldots & \ldots & \ldots \\
                    0 & \ldots & 0 & 0 & 1 & -1 \\
        \end{array} \right)
\end{displaymath}
Clearly, $A_1 \m \dots \m A_k - \Delta(B) = \mathbf{Im} ( \M|_{A_1 \m \dots \m A_k \m B})$.
Further, non--degenerate transformations of lines does not change the cardinality of the image.
Thus, subtracting the $(m+1)$th line, we obtain vectors of the form
$$
    (a_1-a_{m+1}, \dots, a_{m}-a_{m+1}, a_{m+1}- b, \dots, a_k - b)\,,
$$
which belong to
$
     (A_1  \m \dots \m A_{m} - \Delta(A_{m+1}) ) \m (A_{m+1} \m \dots \m A_{k} - \Delta(B) ) \,.
$

To obtain (\ref{f:Ruzsa_triangle''}) it is sufficient to show that
$$
   |Y \m Z - \Delta(X)| \le |Y\m X - \Delta(Z)| \,.
$$
But the  map
$$
    (y_1-x,\dots,y_k-x,z-x) \mapsto (y_1-z,\dots,y_k-z,x-z) \,,
$$
where $(y_1,\dots,y_k) \in Y$, $x\in X$, $z\in Z$
is an injection.
This completes the proof.
$\hfill\Box$
\end{proof}

\begin{remark}
    The proof of the theorem above gives another way to obtain formula (\ref{f:E_k_&_sigma_k}) of Lemma \ref{l:sigma_k}.
    Indeed for any $k\ge 2$ by (\ref{f:higher_sumsets_and_E_k+1}) the following holds
    $$
        |A|^{2k} \le \E_k (A) \cdot |A^{k-1} - \Delta(A)|
    $$
    and we just need to estimate $|A^{k-1} - \Delta(A)|$ in terms of the set $D:=A-A$.
    Such bounds were obtained in \cite{sv} (see Lemma 2.6) but here we use another arguments.
    The cardinality of the set $A^{k-1} - \Delta(A)$ equals the number of tuples
    $$
        (a_1-a_2, a_2-a_3, \dots, a_{k-1} - a_k) = (x_1,\dots,x_{k-1}) \in D^{k-1} \,,
    $$
    where $a_j \in A$, $j\in [k]$.
    Thus
    $$
        |A^{k-1} - \Delta(A)| \le \sum_{x_1,\dots,x_{k-1}} \, \prod_{j=1}^{k-1} \, \prod_{l=0}^{k-1-j} D(x_j+x_{j+1}+\dots+x_{j+l})
            \le
    $$
    $$
            \le
                    \sum_{x_1,\dots,x_{k-1}} D(x_1) \dots D(x_{k-1}) D(x_1+\dots+x_{k-1})
                        =
                \sigma_k (D)
    $$
    and the result follows.
\end{remark}

\bigskip
\noindent As an immediate consequence of (\ref{f:Ruzsa_triangle1}), (\ref{f:Ruzsa_triangle2}) we get
\begin{equation}\label{f:Ruzsa_triangle-classic}
|A^k-\D(A)||A|\ls |A^{k+1}+\D(A)|\,,
\end{equation}
and
$$|A^k+\D(A)||A|\ls |A^k-\D(A)||A+A|\,.$$
In view of (\ref{f:higher_sumsets_and_E_k+1}) we can formulate the following.

\begin{corollary}\label{c:A+B and E_k} Let $A$ and $B$ be finite subsets of an abelian group. Then
$$|A+B|\ge \frac{|A|^2|B|^{1/k}}{\E_k(A)^{1/k}}.$$
\end{corollary}



Let us also remark that
    the proof of Theorem \ref{t:Ruzsa_triangle} prompt to consider different matrices not necessary
  the   matrix $\M$. The type of matrices we used  appears naturally in  studying  $\E_k$.

There is another way to prove estimate (\ref{f:Ruzsa_triangle2}) in
spirit of Lemma 2.4 and Corollary 2.5 from \cite{sv}. We recall this
result.

\begin{proposition}
    Let $k \ge 2$, $m \in [k]$ be positive integers, and
    let $A_1,\dots,A_k,B$ be finite subsets of an abelian group.
    Then
    \begin{equation}\label{f:characteristic1}
        A_1 \m \dots \m A_k - \Delta(B) = \{ (x_1,\dots,x_k) ~:~ B \cap (A_1-x_1) \cap \dots \cap(A_k - x_k) \neq \emptyset \}
    \end{equation}
    and
    \begin{equation}\label{f:characteristic2}
        A_1 \m \dots \m A_k - \Delta(B)
            =
    \end{equation}
    $$
                \bigcup_{(x_1,\dots,x_m) \in A_1 \m \dots \m A_m - \Delta(B)}
                    \{ (x_1,\dots,x_m) \} \m (A_{m+1} \m \dots \m A_k - \Delta(B \cap (A_1-x_1) \cap \dots \cap (A_m-x_m)) \,.
    $$
\label{p:characteristic}
\end{proposition}

\noindent From (\ref{f:characteristic1}) one can deduce another
characterization of the set
$A^k-\D( B)$.
 $$
        A^k-\D( B) = \{ X\subseteq \Gr ~:~ |X| = k,\, B\not\sbeq  ( (\Gr\setminus A)-X)  \,\} \,.
    $$
Here we used $X$ to denote a multiset and a corresponding sequence
created from $X$.
Using the characterization it is easy to prove, that if  $A$ is a subset of finite abelian group $\Gr$ then there is $X$,
$|X| \sim \frac{N}{|A|} \cdot \log N$ such that $A+X = \Gr$.
Indeed, let $A^c = \Gr \setminus A$, and $k \sim \frac{N}{|A|} \cdot \log N$.
Consider
$$|(A^c)^k - \Delta(A^c)| \le |A^c|^{k+1} = N^{k+1} (1-|A|/N)^{k+1} < N^k \,.$$
Thus, there is a multiset $X$, $|X| = k$ such that $A^c \subseteq A-X$.
Whence
the set $-X\cup \{0\}$ has the required property.

Let $D_k (A)$, $S_k(A)$ stand for the
cardinalities of $A^k-\D( A)$, $A^k+\D( A)$, respectively. Next
result describes dependencies between  $D_k(A)$, $S_k(A)$
for different $k$.

\begin{proposition}
    Let $n,m \ge 1$ be positive integers, and $A\subseteq \Gr$ be a finite set.
    Then
    \begin{equation}\label{f:D_n}
        D_n (A) |A|^m \le D_{n+m} (A) \le D_n (A) D_m (A) \,,
    \end{equation}
    and
    \begin{equation}\label{f:S_n}
        S_n (A) |A|^m \le S_{n+m} (A) \le S_n (A) \cdot \min\{ S_m (A),D_m (A) \} \,.
    \end{equation}
    Finally, for $m\ge 2$, we have
    \begin{equation}\label{f:S_n_D_n_1}
        D_n (A) |A|^m \le S_{n+m} (A)\,,
    \end{equation}
    and for $m=1$, $n\ge 2$, we get
    \begin{equation}\label{f:S_n_D_n_2}
        D_{n-1} (A) |A|^2\le  S_{n+1} (A) \,.
    \end{equation}
\label{p:D_n}
\end{proposition}
\begin{proof}
    The first inequality of (\ref{f:D_n}) follows  from  (\ref{f:Ruzsa_triangle1}).
    The second one is a consequence of (\ref{f:Ruzsa_triangle2}) or Proposition \ref{p:characteristic}.
    The first inequality of (\ref{f:S_n}) follows from
     (\ref{f:Ruzsa_triangle1})
    and  (\ref{f:Ruzsa_triangle''}).
    To establish the
    second
    inequality of (\ref{f:S_n}) we use Proposition \ref{p:characteristic}.
    We have
    \begin{equation}\label{tmp:30.07.2011_1}
        S_{n+m} (A) = \sum_{(x_1,\dots,x_m) \in A^{m}+\D( A)} |A^{n}+\Delta (A \cap (x_1-A) \cap \dots \cap (x_m-A))| \,.
    \end{equation}
    Trivially,
    $$
        |A^{n}+\Delta (A \cap (x_1-A) \cap \dots \cap (x_m-A))| \le \min\{ S_n (A), D_n (A) \} \,.
    $$

    It  remains to prove (\ref{f:S_n_D_n_1}), (\ref{f:S_n_D_n_2}).
    By  (\ref{f:Ruzsa_triangle1})  we have  $|A^{n+1}-\Delta(B)| \ge D_n(A) |B|$ for every set $B$.
    Thus, using (\ref{tmp:30.07.2011_1}) once again, we get
    $$
        S_{n+m} (A) \ge D_n (A) \cdot  \sum_{(x_1,\dots,x_{m-1}) \in A^{m-1}+\D( A)} |A \cap (x_1-A) \cap \dots \cap (x_{m-1}-A)|
            = D_n(A) |A|^m \,,
    $$
    provided that  $m\ge 2$.
    Similarly, if $m=1$, $n\ge 2$ then
    $$
        S_{n+1} (A) \ge D_{n-1} (A) \cdot \sum_{x\in A+A} |A \cap (x-A)| = D_{n-1} (A) |A|^2 \,.
    $$
    This completes the proof.
$\hfill\Box$
\end{proof}

\begin{remark}
    It is easy to see that
    all inequalities in Proposition \ref{p:D_n}
    are sharp up to constant factors.
    For example,  if $n,m\ge 2$ then one can consider  $A$
    to be
    a multiplicative subgroup of $\F_p$ or a convex subset of $\R.$
    In this case $D_k, S_k \sim |A|^{k+1}$, for $k\ge 3$ and $|A|^3 \gg D_2, S_2 \gg |A|^3 / \log |A|$ (see \cite{ss,ss-convex,sv})
    and the lower bounds of Proposition \ref{p:D_n} attained for large $n$.
    If $m,n$ are arbitrary then let  $A$ be an arithmetic progression in  $\Z$ or a subspace of $\Z^n_p$.
    We know by (\ref{f:Ruzsa_triangle'}) that $|A|^k \le D_k \le |A-A|^k$, $|A|^k \le S_k \le |A+A|^k$
    hence  all bounds in  Proposition \ref{p:D_n}
    are sharp.
    Nevertheless, if $A\sbeq \Z$, we have always $D_k, S_k \ge (k+1) |A|^k - O_k (|A|^{k-1}),$
    which is a consequence of
    the  trivial inequality $|P+Q| \ge |P|+|Q|-1$, where $P,Q\subseteq \Z$ are arbitrary sets.
\end{remark}

\noindent Proposition \ref{p:D_n}  allows us to introduce a hierarchy of basis
of abelian groups, i.e. of sets $B$ such that $B\pm B = \Gr$. For simplicity, if $B$ is a basis let us write
$B\oplus_kB$ and $B\ominus_k B$ for $B^k+\D(B)$ and $B^k-\D(B),$ respectively.

\begin{definition}
    Let $k\ge 1$ be a positive integer.
    A subset $B$ of an abelian group ${\bf G}$ is called {\it basis of depth $k$}
    if $B \ominus_k B= {\bf G}^k$.
\end{definition}

    It follows from Theorem \ref{t:Ruzsa_triangle} that
 if   $B$ is  a basis of depth $k$ of finite abelian group $\Gr$,
    then for every set $A\subseteq \Gr$
\begin{equation}\label{f:diff-bases}
    |B + A| \ge |A|^{\frac{1}{k+1}} |\Gr|^{\frac{k}{k+1}} \,.
\end{equation}

An analogous inequality for sum bases
will be given in
section \ref{sec:eigenvalues&E_k}.

Inequality (\ref{f:diff-bases}) is trivial if $|B| \ge |A|^{\frac{1}{k+1}} |\Gr|^{\frac{k}{k+1}}$.
In this situation one can use  (\ref{f:D_n}) of Proposition \ref{p:D_n}, which
for any $m\ge k$ gives the following
\begin{equation}\label{f:diff-bases_m}
    |B + A| \ge |B|^{\frac{m-k}{m+1}} |A|^{\frac{1}{m+1}} |\Gr|^{\frac{k}{m+1}} \,.
\end{equation}

Taking any one--element $A$ in formula (\ref{f:diff-bases}) we
obtain, in particular, that $|B| \ge |\Gr|^{\frac{k}{k+1}}$ for any
basis of depth $k$. It is easy to see, using Proposition
\ref{p:characteristic} that every  set  with $B$, $|B| > (1-1/(k+1))
|\Gr|$ is a basis of depth $k$ and this inequality is sharp.
If $S_1,\dots, S_k$ are any sets such that $S_1+\dots+S_k = \Gr$ then
the set $\bigcup_{j=1}^k (\sum_{i\neq j} (S_i - S_i))$ is a basis of depth $k$
(see Corollary \ref{c:G_bases} below, the construction can be found in \cite{Lev+_universal}).
Let us give another example.
Using Weil's bounds for exponential sums
we show that quadratic residuals in $\Z/p\Z$,
    for a prime $p,$ is a basis of depth $(\frac{1}{2}+o(1))\log p$.
Clearly, the bound is the best possible up to constants for subsets of $\Z/p\Z$ of the cardinality less than $p/2$.

\begin{proposition}
    Let $p$ be a prime number, and let $R$ be the set of quadratic residuals.
    Then $R$ is the bases of depth $k$, where $k 2^{k} < \sqrt{p}$.
\label{p:quadratic_residuals}
\end{proposition}
\begin{proof}
Clearly,
$$
    R(x) = \frac{1}{2} \left( 1+ \binom{x}{p} \right) \,,
$$
where $\binom{x}{p}$ is the Legendre symbol.
Put $\a_0 = 0$.
For all distinct non--zero $\a_1,\dots,\a_k$, we have
\begin{eqnarray*}
    |R\cap (R-\a_1) \cap \dots \cap (R-\a_k)|
        &=&
            \frac{1}{2^k} \sum_x \prod_{j=0}^k \left( 1+ \binom{x+\a_j}{p} \right)
                \ge
                    \frac{1}{2^k} \left( p - \sqrt{p} \cdot \sum_{j=2}^k j C^j_k
                    \right)\\
    &\ge&
        \frac{1}{2^k} \left( p - \sqrt{p} \cdot k 2^k \right) > 0 \,.
\end{eqnarray*}
We used the well--known Weil bound for exponential sums with
multiplicative characters (see e.g. \cite{Johnsen}). By
(\ref{f:characteristic1}) in Proposition \ref{p:characteristic} we
see that $R\ominus_k R = \Z_p^k$. $\hfill\Box$
\end{proof}

\bigskip

Another consequence of Proposition \ref{p:quadratic_residuals}
is that quadratic non--residuals $Q$ (and, hence, quadratic residuals)
have no completion of size smaller then $(\frac12+o(1))\log p$,
that is a set $X$ such that $X+Q = \Z/p\Z$.

\bigskip

The next proposition is due to N.G. Moshchevitin.

\begin{proposition}
    Let $k_1,k_2$ be positive integers, and $X_1,\dots,X_{k_1},Y$, $Z_1,\dots,Z_{k_2},W$ be finite subsets of an abelian group.
    Then we have a bound
    $$
        |X_1 \m \dots \m X_{k_1} - \Delta(Y)| |Z_1 \m \dots \m Z_{k_2} - \Delta(W)|
            \le
    $$
    $$
            \le
                | (X_1 - W) \m \dots \m (X_{k_1}- W) \m (Y-Z_1) \m \dots \m (Y-Z_{k_2}) - \Delta(Y-W)| \,.
    $$
\end{proposition}
\begin{proof}
It is enough to observe that the map
$$
    (x_1 - y,\dots,x_{k_1}-y, z_1-w, \dots, z_{k_2} - w)
        \mapsto
$$
$$
        \mapsto
            (x_1-w-(y-w),\dots, x_{k_1} - w - (y-w), y-z_1 - (y-w), \dots, y-z_{k_2} - (y-w))
$$
where $x_j\in X_j$, $j\in [k_1]$, $y\in Y$, $z_j\in Z_j$, $j\in [k_2]$, $w\in W$ is  injective.
$\hfill\Box$
\end{proof}

\bigskip
\noindent In particular, the difference and the sum of two bases
of depths $k_1$ and $k_2$ is a basis of depth $k_1+k_2$. Let us also
formulate
a simple identity, which is a consequence
of Theorem \ref{t:Ruzsa_triangle}.

\begin{corollary}
    Let $k \ge 2$ be a positive integer, and
    let $A_1,\dots,A_k$ be a subsets of a finite abelian group $\Gr$.
    Then
    \begin{equation}\label{}
        |A_1  \m \dots \m A_k - \Delta(\Gr)| = |\Gr| |A_1  \m \dots \m A_{k-1} - \Delta(A_k)| \,.
    \end{equation}
\label{c:G_bases}
\end{corollary}
\noindent Thus, $B$ is a basis of depth $k$ iff $B$ is
$(k+1)$--universal set (see \cite{abs}), i.e. a set that is for any
$x_1,\dots,x_{k+1} \in \Gr$ there is $z\in \Gr$ such that
$z+x_1,\dots, z+x_{k+1} \in B$.
A series of very interesting  examples of universal sets can be found in \cite{Lev+_universal}.

Finally,   we also formulate an interesting  consequence of  the inequality (\ref{f:S_n}).

\begin{corollary}\label{c:sum-diff-bases}
    Let $k>m\ge 1$ be integers and
    let  $B \subseteq \Gr$ be a set such that $B\oplus_k B = \Gr^k$.
    Then $B$ is a basis of depth $m$, that is $B\ominus_m B = \Gr^m$.
\end{corollary}

An inverse theorem to Corollary \ref{c:sum-diff-bases} is related to a known problem:
does there exist an integer $n$ such that if $A-A=\Gr$ then $nA=\Gr?$
It was answer in the negative in \cite{Haight}.
 However,
 it is easy to see
 that such a constant exists provided $A$ is a basis of sufficiently high depth.

\begin{proposition}\label{p:diff-sum-bases}
Let $B$ be a basis of depth $k$ of a finite abelian group $\Gr$ of density $\d.$ Then
$nB=\Gr$ for every
$$n
    \ge
3
+ \frac{2}{\log (k+1)} \log \left( \frac{\log (1/\delta)}{\log ((k+1)/2) } \right) \,.$$
\end{proposition}
\begin{proof} We will use an elementary fact that if $X,Y\sbeq G$ then there exists $x$ such that
\begin{equation}\label{e:eveint}
|(X+x)\cap Y|\le |X||Y|/N.
\end{equation}
Now
prove that for every set $A\sbeq \Gr$ we have $|A+B|\ge \min((k+1)|A|/2,N/2).$
Indeed, applying iteratively (\ref{e:eveint}), there exists  a set $S$ of size $k$ such that $|A+S|\ge \min ( (k+1)|A|/2,N/2).$
Since $B\ominus_k B = \Gr^k$ it follows that there is $a\in B$ such that $S+a\sbeq B,$ so that
 $|B+A|\ge \min ((k+1)|A|/2,N/2).$
Therefore, for every $s\ge 1$
\begin{equation}\label{f:sA_1}
    |sB| > \min(((k+1)/2)^s |B|,N/2) \,
\end{equation}
On the other hand, using  (\ref{f:diff-bases}) iteratively, we get
\begin{equation}\label{f:sA_2}
    |lB| \ge \d^{\frac{1}{(k+1)^l}} N
\end{equation}
for all positive integers $l$.
Combining, (\ref{f:sA_1}), (\ref{f:sA_2}) and optimizing over $s$, $l$,
we have for
$$
    t
        >
            \frac{1}{\ln (k+1)} + \frac{\log \ln (k+1)}{\log (k+1)} +
                \frac{1}{\log (k+1)} \log \left( \frac{\log (1/\delta)}{\log ((k+1)/2) } \right)
$$
that $2tA=\Gr.\hfill\Box$
\end{proof}

\section{Croot-Sisask Lemma}
\label{sec:CS}

\noindent Croot and Sisask \cite{cs} proved the following remarkable
result, which found many deep applications, see \cite{sanders1}, \cite{sanders2}.
We formulate their result in a simple form.

\begin{theorem}{\rm (Croot--Sisask)} Let $A,B$ be subsets of a group and $k\in \N.$ Suppose that $|A-A|\ls K|A|.$
Then there exists  $T\sbeq A$ such that $|T|\gs |A|/(2K)^k $ and
$$\|(A* B)(x)-(A* B)(x+t)\|_2^2\ls 8|A|^2|B|/k$$
for every $t\in T.$
\end{theorem}

\no We prove that if the energy $\E_k(A)$ is not much larger than $
|A|^{k+1}/K^{k-1}$ then one can substantially improve the lower
bound on  size of the set of almost--periods $T$ (provided that $\Gr$
is abelian).

\begin{theorem}\label{t:croot-sisask}
 Let $A,B$ be subsets of an abelian group and $k\in \N.$
Suppose that $|A-A|\le K|A|$ and $\E_{2k+2}(A)=M|A|^{2k+3}/K^{2k+1}.$
Then there exists  $T\sbeq A-A$ such that $|T|\gs K|A|/(16M) $ and
$$\|(A* B)(x)-(A*B)(x+t)\|_2^2\ls 32|A|^2|B|/k$$
for every $t$ belonging to a shift of  $T.$
\end{theorem}
\begin{proof} We choose  uniformly at random a $k$--element sequence $X=(x_1,\dots,x_k),~x_i\in  A.$
As in the proof Croot--Sisask theorem we say that $X$ {\it approximates} $A$ if
$$\|(\mu_X* B)(x)-(A*B)(x)\|_2^2\ls 2|A|^2|B|/k\,,$$
where $\mu_X(x)=X(x)\cdot |A|/k$ (by $X$ we mean the characteristic
function of the set $\{x_1,\dots,x_k\}$). Following Croot-Sisask
argument   we have
\begin{equation}\label{markov}
{\mathbb P}(X \text{ approximates } A)\gs 1/2.
\end{equation}

For $s\in A^{k}-\D(A)$ let $A'_s$ be the set of all $a\in A$ such
that $s+\Delta(a)\sbeq A^{k}$ and $s+\Delta(a)$ approximates $A.$ Then
$$\|(\mu_{\Delta(a)+s}*B)(x)-(A*B)(x)\|_2^2\ls 2|A|^2|B|/k$$
for every $s$ and  $a\in A'_s.$ Therefore, by the triangle
inequality we have
\begin{equation}\label{f:period}
\|(A* B)(x)-(A*B)(x+a)\|_2^2\ls 8|A|^2|B|/k
\end{equation}
for every
$a$ belonging to a shift of $A'_s.$
 By the Cauchy-Schwarz inequality
$$|A'_s||A'_t|\ls \E(A'_s,A'_t)^{1/2}|A'_s-A'_t|^{1/2}\,.$$
Again using the  Cauchy-Schwarz inequality and Lemma
\ref{l:E_k-identity} we get
$$\Big (\sum_{s,t\in A^{k}-\D(A)}|A'_s||A'_t|~\Big)^2\ls \E_{2k+2}(A)\sum_{s,t\in A^{k}-\D(A)}|A'_s-A'_t|\,.$$
By (\ref{markov})
$$\sum_{s\in A^{k}-\D(A)}|A'_s|\gs (1/2){|A|}^{k+1}\,,$$
so that
$$(1/16)K^{2k+1}M^{-1} |A|^{2k+1}\ls \sum_{s,t\in A^{k}-\D(A)}|A'_s-A'_t|\ls |A^{k}-\D(A)|^2\max |A'_s-A'_t|\,.$$
Thus,  there exist $s_0$ and $t_0$ such that $|A'_{s_0}-A'_{t_0}|\gs
K|A|/(16M).$ To finish the proof it is enough to use
(\ref{f:period}) for $s_0$ and $t_0$  and apply the triangle
inequality.  The assertion is satisfied for  a shift of $A'_{s_0}-A'_{t_0}.$
$\hfill\Box$
\end{proof}

\section{Small higher  energies and the structure of sets}
\label{sec:E_3_energy_&_structure}

The aim of this section is to prove that small $\E_3(A)$ implies the
existence of a  large very structured subset of $A.$ We make use of
the following lemma (see \cite{ss-convex}).

\begin{lemma}\label{l:pop}
    Let $A$ be a subset of an abelian group, $P_* \subseteq A-A$ and $\sum_{s\in P_*} |A_s| = \eta |A|^2$, $\eta \in (0,1]$.
    Then
    \begin{equation*}
        \sum_{s\in P_*} |A\pm A_s| \ge \eta^2 |A|^6\E^{-1}_3(A)  \,.
    \end{equation*}
\end{lemma}

\noindent The next lemma is the well--known Balog--Szemer\'edi--Gowers
theorem.

\begin{lemma}\label{l:bsg}
Let $A$ and $B$ be finite sets of an abelian group, and $|A|\gs
|B|$. If $\E(A,B)=\a |A|^3,$ then there exist sets $A'\sbeq A$ and
$B'\sbeq B$ such that $|A'|,|B'|\gg \a |A|$ and
$$|A'+B'|\ll \a^{-5}|A|\,.$$
\end{lemma}

\noindent For a set $A$ denote by $P=P(A)$ the set of all elements
in $A-A$ that have at least $|A|^2/(2|A-A|)$ representations.

\begin{theorem}\label{t:small E_3} Let $A$ be a subset of an abelian group.
Suppose that $|A-A|= K|A|$ and $\E_{3}(A)=M^{}|A|^{4}/K^{2}.$ Then
there exists $A'\sbeq A$ such that $|A'|\gg |A|/M^{5/2}$
 and
 $$|nA'-mA'|\ll M^{12(n+m)+5/2}K|A'|$$
 for every $n,m\in \N.$
\end{theorem}

\begin{proof} Put $D=A-A$ and let $P=P(A)$.
Clearly, $M\ge 1$.
Then
$$\sum_{s\in P} (A\c A)(s)\gs \frac12 |A|^2 \text{~~~and~~~} \sum_{s\in P} (A\c A)(s)^3 \gs \frac12\E_3(A)\,.$$
By the H\"older inequality
$$|A|^2\ll \E_3(A)^{1/3} |P|^{2/3}\,,$$
so that $|P|\gg |D|/M^{1/2}.$

By the Katz--Koester transform (see \cite{kk}), we have $A-A_s\sbeq
D\cap (D+s).$ Using the Cauchy--Schwarz inequality and Lemma
\ref{l:pop}, we obtain
\begin{eqnarray}\label{f:energy-thm20}
\E(D,P)&=&\sum_{s,s'\in P} |(D+s)\cap (D+s')|
\gs |D|^{-1}\Big(\sum_{s\in P}|D\cap (D+s)|\Big)^2 \\
&\gs& |D|^{-1}\Big(\sum_{s\in P}|A-A_s|\Big)^2 \gg |D|^3/M^{2} \,.
\end{eqnarray}
Hence
by Lemma \ref{l:bsg} there are sets $D'\sbeq D, ~P'\sbeq P$ such
that $|D'|\gg M^{-2}|D|,~|P'|\gg M^{-2}|P|$ and
$$|D'+P'|\ll M^{12}|D'| \,.$$
Pl\"unnecke--Ruzsa inequality (see e.g. \cite{tv}) yields
\begin{equation}\label{plunnecke}
|nP'-mP'|\ll M^{12(n+m)}|D'|\ll M^{12(n+m)+5/2}|P'|,
\end{equation}
for every $n,m\in \N.$ By pigeonhole principle there is $x$ such
that
$$|(A-x)\cap P'|\gg |P'| /(2K)\gg |A|/M^{5/2}\,.$$
Put $A'=A\cap (P'-x)$. Thus, by (\ref{plunnecke}) and the previous inequality,
we get
$$|nA'-mA'|\ll M^{12(n+m)+5/2}|P'|\ll M^{12(n+m)+5/2}K|A'|$$
for every $n,m\in \N.$ $\hfill\Box$
\end{proof}

\bigskip

\no Observe that if $A'$ is a set given by Theorem \ref{t:small E_3}
then
$$\frac{M|A|^4}{K^2}=\E_3(A)\gs \E_3(A')\gs \frac{|A'|^6}{|A'-A'|^2}\,,$$
hence
$$|A'-A'|\gg M^{-O(1)}K|A'|\,.$$
 Therefore, by Theorem \ref{t:small E_3}, we obtain
 $$|(A'-A')+(A'-A')|\ls M^{O(1)}|A'-A'|\,.$$
 Applying Sanders theorem \cite{sanders1} for $A'-A'$,
 we obtain that $A'-A'$ is contained in an generalized arithmetic progression of
 dimension $ M^{O(1)}$ and size $Ke^{M^{O(1)}}|A'|$. In particular, $4A'-4A'$ contains an
 arithmetic progressions of length $|A'|^{(\log M)^{-O(1)}}.$

Recall that a subset $\L = \{ \l_1, \dots, \l_t \}$ of a finite Abelian
group $\Gr$ is called {\it dissociated} if
$\sum_{j=1}^t \eps_j \l_j = 0 $, where $\eps_j \in \{0,-1,1\}$
implies $\eps_j = 0$, $j\in [t].$ For a set $Q \subseteq \Gr$ let
$\dim(Q)$ denote the size of the largest dissociated subset of $Q$.

In \cite{sy} the following result was proved.

\begin{theorem}\label{t:energy-dim}
    Let $\Gr$ be a finite Abelian group, $A,B \subseteq \Gr$ be two sets, and $c\in (0,1]$.
    Suppose $\E (A,B) \ge c |A| |B|^2;$ then there exist a set $B_1 \subseteq B$
    such that $\dim (B_1) \ll c^{-1} \log |A|$
    and
    \begin{equation}\label{f:E(A,B_1)}
        \E (A,B_1) \ge 2^{-5} \E(A,B).
    \end{equation}
    In particular, $|B_1| \ge 2^{-3} c^{1/2} |B|$.
    If $B=A$ then $\E(B_1) \ge 2^{-10}\E(A)$ and, consequently, $|B_1| \ge 2^{-4} c^{1/3} |A|$.
\end{theorem}
\label{t:E(A,B)}

We supplement Theorem \ref{t:small E_3} with the following
statement.

\begin{corollary}
    Let $A$ be a subset of an abelian group.
Suppose that $|A-A|= K|A|$ and $\E_{3}(A)=M^{}|A|^{4}/K^{2}.$ Then
there exists $A_* \sbeq A$ such that $|A_*|\gg |A|/M$
 and
$$
    \dim (A_*) \ll M^2 (\log |A|+\log K) \,.
$$
\end{corollary}
\begin{proof}
By (\ref{f:energy-thm20}), we have
$$
    \E(D,P) \gg \frac{K^3 |A|^3}{M^2}
$$
and by Theorem \ref{t:energy-dim} there exists $P_*\sbeq P$ with $\dim(P_*)\ll M^2 (\log |A|+\log K)$ such that
$$
    |P_*|^2 |D| \ge \E(D,P_*) \gg \frac{K^3 |A|^3}{M^2}
$$
Thus $|P_*|\gg K|A|/M$.
Again for some $x$ we have $|A\cap (P_*-x)|\gg |P_*| / K \gg |A|/M$, so that the assertion
follows for $A_*=A\cap (P_*-x)$. $\hfill\Box$
\end{proof}

\section{Bounding energies in terms of $|AA|$}
\label{sec:product_E_k}

Let  $A\sbeq \R$  and let $AA=\{ab\,:\, a,b\in A\}$ and $A/A=\{a/b: a,b\in A, b\not=0\}.$ Denote by
$\E^{\times}_k(A)$ the multiplicative  energy of order $k$.
Solymosi \cite{sol} using ingenious argument proved that
$$\E^\times (A)\ll |A+A|^2\log |A|$$
for
every
set of real numbers $A$.

In this section we prove some sum-product type estimates.
Our basic tool is the following Lemma \ref{l:pop_SzT},
which is a generalization of Lemma 2.6 in \cite{ss-convex} and an
 improvement of  Lemma 4.1 in \cite{Li}. We will make use of Szemer\'edi--Trotter theorem
\cite{sz-t}. We call a set ${\mathcal L}$ of continuous
plane curve a {\it pseudo-line system} if any two members of ${\mathcal L}$
share at most one point in common.

\begin{theorem}\label{thm:szemeredi-trotter} (\cite{sz-t})
Let $\P$ be a set of points and let ${\mathcal L}$ be a pseudo-line system.
Then
$$\I(\P,{\mathcal L})=|\{(p,l)\in \P\times {\mathcal L} : p\in l\}|\ll |\P|^{2/3}|{\mathcal L}|^{2/3}+|\P|+|{\mathcal L}|.$$
\end{theorem}

\begin{lemma}\label{l:pop_SzT}
 Let $A,B,C$ be subsets of reals and let $f$ be a strictly convex  function.
Suppose that $|A+B|\le M|B|$. Then
$$|\{x\in f(A)+C: (f(A)*C)(x)\ge \t\}|\ll (M\log M)^2\frac{|B||C|^2}{\t^3}.$$
\end{lemma}
\begin{proof} Obviously, it is enough to prove the assertion for $1\ll \t
\ls \min \{|A|,|C|\}$.
 For real numbers  $\a,\b$ put
$l_{\a,\b}=\{(q,f(q)): q\in A\}+(\a,\b).$ We consider  the
pseudo-line system ${\mathcal L}=\{l_{\a,\b}\,:\, ~\a\in B,\, \b \in
C\},$ and the set  of points $\P=(A+B)\times (f(A)+C).$ Let $\P_\t$ be the set of points of $\P$
belonging to at least $\t$ curves from ${\mathcal L}$.
Clearly, $|{\mathcal L}|
=|B||C|$ and $\I (\P_\t,{\mathcal L})\gs \t |\P_\t|$.

 By Szemer\'edi-Trotter's theorem we have
\begin{equation}\label{eq:s-t}
\t |\P_\t|\ll (|\P_\t||B|||C|)^{2/3}+|B||C|+|\P_\t|,
\end{equation}
so  that $|\P_\t|\ll |B|^2|C|^2/\t^3.$

Now suppose that $(f(A)*C)(x)\ge \t.$ Let $X$ be the set of all $a\in A$ such that there exists $c\in C$
with $f(a)+c=x.$ Clearly $|X|=(f(A)*C)(x)$ and
$$\sum_{s\in X+B}(X*B)(s)= |X| |B|,$$
so that there is $0\le i=i(x) \le \log M$ such that
$$\sum_{2^{i-1}\t/M\le (X*B)(s)\le 2^{i}\t/M}(X*B)(s)\ge \frac{\t |B|}{2\log(2M)}.$$
Hence  each $x$ with $(f(A)*C)(x)\ge \t$
gives at least $M|B|/2^{i(x)+1}\log(2M)$ points $p\in \P_{\t2^{i(x)-1}/M}$ having the same ordinate.
Furthermore, for at least  $|\{x: (f(A)*C)(x)\ge \t\}|/\log (2M)$
elements $x$ we have the same choice for $i(x)=i_0.$ Thus, we have
$$\frac{M|B|}{2^{i_0}\log M}\frac{|\{x:(f(A)*C)(x)\ge \t\}|}{\log M}\ll |\P_{2^{i_0-1}\t/M}|.$$
In view of
$$|\P_{2^{i_0-1}\t/M}|\ll M^3|B|^2|C|^2/2^{3i_0}\t^3,$$
we infer that
$$|\{x: (f(A)*C)(x)\ge \t\}|\ll (M\log M)^2\frac{ |B||C|^2}{\t^3}\,.~~~\hfill\Box$$
\end{proof}
\bigskip

Order elements  $s\in A-A$  such that $(A\circ A)(s_1)\ge (A\c
A)(s_2)\ge\dots \ge (A\c A)(s_t),\, t=|A-A|.$ Taking in Lemma
\ref{l:pop_SzT}, $A=B:=\log A$ (if necessary we consider $A_+$ or
$(-A_-)$), $C:=A$ and
 $f=\exp,$ we obtain the following bound.

\begin{corollary}\label{lcon}
Suppose that $A\sbeq \R$ and $|AA|\le M|A|$. Then for every $r\ge 1$ we have
$$(A\c A)(s_r)\ll (M\log M)^{2/3}|A|/r^{1/3}.$$
\end{corollary}

Thus, we have
$$\E(A)\ll |AA||A|^{3/2}\log |A|$$
and
$$\E_k(A)\ll |AA|^{2k/3}|A|^{k/3}(\log |A|)^{O(k)}$$
for every $k\ge 3.$
One can improve the above bounds for a dense subset of $A$ provided that $\E^\times_3(A)$ is small.

\begin{corollary}\label{cor-multienergy}
Suppose that $A\sbeq \R$ and $\E^\times_3(A)\le M|A|^{6}/|A/A|^{2}.$ Then
there exists a set $A'\sbeq A$ such that $|A'|\ge |A|/M^{O(1)}$ and
$$\E(A')\ll M^{O(1)}|AA|^{1/2}|A|^2,$$ and
$\E_k(A')\ll M^{O(k)}|AA|^{k/3}|A|^{2k/3}(\log |A|)^{O(k)}$
for every  $k\ge 3.$
\end{corollary}
\begin{proof} By  Theorem \ref{t:small E_3} there is a set $A'\sbeq A$ such that
$|A'|\ge |A|/M^{O(1)}$ and $|A'A'A'|\le M^{O(1)}|A/A|.$ Furthermore,
$$M^{O(1)}\frac{|A|^4}{|A/A|}\ge \E^\times (A)\ge \E^\times (A')\ge \frac{|A'|^4}{|A'A'|},$$
so that
$|A'A'|\ge |A/A|/M^{O(1)}$ and  $|A'A'A'|\le M^{O(1)}|A'A'|$.
We apply Lemma \ref{l:pop_SzT} with $A=\log A'$ (if necessary we
consider $A'_+$ or $(-A'_-)$) $, B=\log A'A',$
$C=A, f=\exp.$ Thus
$$|\{x: (A'*A')(x)\ge \t\}|\ll (M\log M)^{O(1)}\frac{|A'A'||A|^2}{\t^3},$$
and the assertion follows.$\hfill\Box$
\end{proof}

\bigskip
We finish this section with some remarks concerning a sum--product kind
result of Balog \cite{Balog_AA+A}.
He proved that for every finite sets $A,B,C,D$ of reals we have
$$|A C+ A||B C+ B|\gg |A||B||C|\,$$
and
$$|A C+ A D||B C+ B D|\gg |B/A||C||D|\,,$$
so, in particular, $|A A+A|\gg |A|^{3/2}$ and $|A A+A A|\gg |A||A/A|^{1/2}$.
However, carefully following his argument one can see that actually he obtained stronger inequalities
$$|(A\times B) \cdot \D(C)+ A\times B|\gg |A||B||C|\,$$
and
$$|(A\times B) \cdot \D(C)+ (A\times B) \cdot \D(D)|\gg |B/A||C||D|\,.$$
Assume for simplicity that $A=B=C$ and put $A^\times_q=A\cap Aq^{-1}.$

\begin{theorem}\label{t:balog-sum-prod} Let $A\sbeq \R$ be a finite set and suppose that $\E_3^\m(A)=M|A|^6/|A/A|^2.$  Then
$$|AA+A|\gg  |A||A/A|^{1/2}M^{-1/2}$$
and
$$|AA+AA|\gg |A/A|^{3/2}M^{-1} \,.$$
\end{theorem}
\begin{proof} We will closely follow Balog's proof, so we only sketch the argument. Let $l_i$ be the line
$y=q_ix.$ Thus, $(x,y)\in l_i\cap A^2$ if and only if $x\in  A^\m_q.$
Let $q_1,\dots,q_n\in A/A$ be such  that $q_1<q_2<\dots<q_n$ and $|A^\m_{q_i}|\ge |A|^2/2|A/A|,$
so that $\sum_i |A^\m_{q_i}|\ge \frac12|A|^2.$ We multiply all points of $A^2$ lying on the line $l_i$ by
$\D(A)$, so we obtain $|AA^\m_{q_i}|$ points still belonging to the line $l_i$ and then we consider sumset of
the resulting set with $l_{i+1}\cap A^2.$ Clearly, we obtain $|AA^\m_{q_i}||A^\m_{q_{i+1}}|$ points from the set $(AA+A)^2$
lying between the lines $l_i$ and $l_{i+1}$. Therefore, we have
$$|AA+A|^2 \ge \sum_{i=1}^{n-1} | A^\times_{q_i}||AA^\times_{q_{i+1}}|\gg \frac{|A|^2}{|A/A|}\sum_{i=1}^{n-1} |AA^\times_{q_{i+1}}|\,,$$
and by Lemma \ref{l:pop}
$$|AA+A|^2\gg \frac{|A|^8}{|A/A|\E^\m_3(A)}\,.$$
To prove the second assertion let  $q_1,\dots,q_n\in A/A$ be such  that $q_1<q_2<\dots<q_n$ and $|AA^\m_{q_i}|\ge |A/A|/2M.$
We multiply  all points of $l_i\cap A^2$ and $l_{i+1}\cap A^2$ by
$\D(A)$ and the we  consider their sumset. We obtain $|AA^\m_{q_{i+1}}||AA^\m_{q_i}|$ points that belong to $(AA+AA)^2.$
By Lemma \ref{l:pop} we have
$$\sum_{q} |AA^\times_{q}|\ge \frac{|A|^6}{\E^\m_3(A)}=\frac{|A/A|^2}{M}\,$$
so that $n\gg |A/A|/M.$
Therefore, it follows that
$$|AA+AA|^2\ge\sum_j|A A^\times_{q_j}||A A^\times_{q_{j+1}}|\gg \frac{|A/A|^3}{M^2}\,,$$
which completes the proof.$\hfill\Box$
\end{proof}

\bigskip

\begin{remark}
    By Proposition \ref{p:characteristic}, we have
    \begin{equation*}\label{f:Prop_10_mult}
        \sum_{q \in A/A} |AA^\times_{q}| = |(A\m A) \cdot \Delta^* (A)| \,,
    \end{equation*}
    where $\Delta^* (A) = \{ (a,a^{-1}) ~:~ a\in A \}$.
    Thus by the averaging argument, one gets
    \begin{equation}\label{f:Prop_10_mult'}
        \sum_{q \in A/A ~:~ |AA^\times_{q}| \ge 2^{-1} |(A\m A) \cdot \Delta^* (A)| / |A/A|}
            |AA^\times_{q}| \ge  2^{-1} |(A\m A) \cdot \Delta^* (A)| \,,
    \end{equation}
    The proof of the Theorem \ref{t:balog-sum-prod} and
    formula (\ref{f:Prop_10_mult'})
    give 
    another
    inequality on,
    namely
    $$
        |AA + AA| \gg \frac{|(A\m A) \cdot \Delta^* (A)|}{|A/A|^{1/2}} \,.
    $$
\end{remark}

\bigskip

We also formulate another consequence of Solymosi's bound for multiplicative energy.

\begin{corollary}\label{c:solymosi-E_3} Let $A\sbeq \R$ be a finite set and suppose that $\E_3(A)=M|A|^6/|A-A|^2.$  Then
$$|A(A+A)|\gg  \frac{|A|^2}{M^{O(1)}\log |A|}\,.$$
\end{corollary}
\begin{proof} By Theorem \ref{t:small E_3} there is a set $A'\sbeq A$
such that $|A'|\gg |A|/M^{O(1)}$ and $|4A'|\ll M^{O(1)}|A-A|$. Moreover, observe that
$$\frac{|A'|^4}{|A'+A'|}\le \E(A')\le \E(A)\le \frac{M^{1/2}|A|^4}{|A-A|}\,,$$
so that $|A'+A'|\ge |A-A|/M^{O(1)}$ and $|4A'|\ll M^{O(1)}|A'+A'|.$ The required estimate follows now from a general
version of Solymosi's result
$$\E^\m(A',A'+A')\le |A'+A'||4A'|\log |A|$$
and the trivial estimate $\E^\m(A',A'+A')\gg |A'|^2|A'+A'|^2/|A'(A'+A')|.$$\hfill\Box$
\end{proof}

\section{Higher energies, eigenvalues and  the magnification ratios }
\label{sec:eigenvalues&E_k}

Let $A,B\subseteq \Gr$ be two finite sets. The magnification ratio
$R_B [A]$ of the pair $(A,B)$ (see e.g. \cite{tv}) is defined by
\begin{equation}\label{f:R_B[A]}
    R_B [A] = \min_{\emptyset \neq Z \subseteq A} \frac{|B+Z|}{|Z|} \,.
\end{equation}
We simply write $R[A]$ for $R_A [A]$. Petridis  \cite{p} obtained an
amazingly short proof of the following fundamental theorem.

\begin{theorem}
    Let $A\subseteq \Gr$ be a finite set, and  $n,m$ be positive integers.
    Then
    $$
        |nA-mA| \le R^{n+m} [A] \cdot |A| \,.
    $$
\label{t:Petridis}
\end{theorem}

\no Another beautiful result (which implies Theorem
\ref{t:Petridis}) was proven also by Petridis \cite{p}.

\begin{theorem}
    For any $A,B,C$, we have
    $$
        |B+C+X| \le R_{B} [A] \cdot |C+X| \,,
    $$
    where $X\subseteq A$ and $|B+X| = R_B [A] |X|$.
\label{t:Petridis_C}
\end{theorem}

\no  For a set  $B\subseteq \Gr^k$
    define
\begin{equation*}\label{}
        R_B [A] = \min_{\emptyset \neq Z \subseteq A} \frac{|B+\Delta(Z)|}{|Z|} \,.
\end{equation*}
In the next two results we assume that $X\sbeq A$ is such that
$|B+\D(X)|=R_B[A]|X|.$ It is easy to see that Petridis argument can
be adopted to higher dimensional sumsets, giving a  generalization
of Theorem \ref{t:Petridis_C}.

\begin{theorem}
    Let $A\subseteq \Gr$ and $B\sbeq \Gr^k.$
    Then for any $C\subseteq \Gr$, we have
    $$
        |B+\Delta(C+X)| \le R_{B} [A] \cdot |C+X| \,.
    $$
    \label{t:Petridis_C_Delta}
\end{theorem}

\no A consequence of Theorem \ref{t:Petridis_C_Delta}, we obtain a
generalization of the sum version of the triangle inequality
(see, e.g. \cite{cochraine_comment}).

\begin{corollary}
    Let $k$ be a positive integer, $A,C\subseteq \Gr$ and $B\subseteq \Gr^k$ be finite sets.
    Then
    $$
        |A| |B+\Delta(C)| \le |B+\Delta(A)| |A+C| \,.
    $$
\label{c:triangle_plus}
\end{corollary}
\begin{proof}
 Using Theorem \ref{t:Petridis_C_Delta}, we have
$$
    |B+\Delta(C)| \le |B+\Delta(C+X)| \le R_B [A] \cdot |C+X| \le \frac{|B+\Delta(A)|}{|A|} |A+C|
$$
and the result follows.
$\hfill\Box$
\end{proof}

\bigskip

\no Thus, we have  the following sum--bases  analog of inequality (\ref{f:diff-bases}).

\begin{corollary}
    Let $k$ be a positive integer, and $B\oplus_k B = \Gr^k$.
    Then for any set $A\subseteq \Gr$, we have
$$
    |B + A| \ge |A|^{\frac{1}{k+1}} |\Gr|^{\frac{k}{k+1}} \,.
$$
\label{c:sum-basis}
\end{corollary}

For an integer $k\ge 1$ define
\begin{equation}\label{f:R_B[A]}
    R^{(k)}_B [A] = \min_{\emptyset \neq Z \subseteq A} \frac{|B^k+\Delta(Z)|}{|Z|} \,,
\end{equation}
where $A,B \subseteq \Gr$. So, $R^{(1)}_B [A] = R_B [A]$. The aim of
this section is to  obtain {\it lower} bounds for $R^{(k)}_B [A]$ in
terms of the energies $\E_{2k+1} (A,B)$. We make use of the
singular--value decomposition lemma (see e.g. \cite{Gowers_groups}).

\begin{lemma}
    Let $n,m$ be two positive integers, $n\le m$, and let $X,Y$ be  sets of cardinalities $n$ and $m$,
    respectively.
    Let also $\M=\M(x,y)\,$, $x\in X$, $y\in Y,$  be $n\m m$ real matrix.
    Then there are functions $u_j : X \to \R$, $v_j : Y \to \R$, and non--negative numbers $\l_j$
    such that
    \begin{equation}\label{f:M_singular_decomposition_basic}
        \M (x,y) = \sum_{j=1}^n \l_j u_j (x) v_j (y) \,,
    \end{equation}
    where $(u_j)$, $j\in [n]$,  and $(v_j)$, $j\in [n]$ form two orthonormal sequences,
    and
    \begin{equation}\label{}
        \l_1 = \max_{w\neq 0} \frac{\| \M w\|_2}{\|w\|_2}\,, \quad \l_2 = \max_{w\neq 0,\, w\perp u_1} \frac{\| \M w\|_2}{\|w\|_2}
            \,, \dots \,, \l_n = \max_{w\neq 0,\, w\perp u_1,\, \dots,\, w\perp u_{n-1}} \frac{\| \M w\|_2}{\|w\|_2} \,.
    \end{equation}
\label{l:singular_decomposition}
\end{lemma}

Next corollary collects  further properties of singular values
$\l_j$ and vectors $u_i$, $v_j$, which we shall use in the course of
the proof of the main result.

\begin{corollary}
    With the notation of the previous lemma, we have \\
    $\bullet$ $\M u_j = \l_j v_j$, $j\in [n]$. \\
    $\bullet$ The numbers $\l^2_j$ and the vectors $u_j$ are all eigenvalues and eigenvectors of the matrix $\M^* \M$. \\
    $\bullet$ The numbers $\l^2_j$ and the vectors $v_j$ form $n$ eigenvalues and eigenvectors of the matrix $\M \M^*$.
        Another $(m-n)$ eigenvalues of $\M\M^*$ equal zero. \\
    $\bullet$ We have $\sum_{j=1}^n \l^2_j = \sum_{x,y} \M^2 (x,y)$, and
                \begin{equation}\label{f:ractangular_norm_of_M}
                    \sum_{j=1}^n \l^4_j = \sum_{x,x'} \Big| \sum_y \M(x,y) \M(x',y)\Big|^2 \,.
                \end{equation}
\label{cor:singular_decomposition}
\end{corollary}
\begin{proof} The first and the last property follows directly from Lemma \ref{l:singular_decomposition}.
To obtain the second and the third statements let us  note that
$$(\M^* \M) (x,y) = \sum_{j=1}^n \l^2_j u_j (x) u_j (y)$$ and similarly for $\M\M^*$.
Thus, by the first formula of the fourth statement,  all another eigenvalues of nonnegative definite matrix $\M\M^*$ equal zero.
$\hfill\Box$
\end{proof}

\bigskip

The quantity  (\ref{f:ractangular_norm_of_M}) is called the {\it
rectangular norm} of $\M$. We denote it by $\| \M \|^4_{\Box}$.
Further properties of $\l_j$, $u_i$, $v_j$ can be found in \cite{Gowers_groups}.

Let $k\ge 1$ be a positive integer, $A,B\subseteq \Gr$ be finite sets,
and put $X=B^k - \Delta(A)$, $Y=A$.
Clearly, $|X| \ge |Y|$.
Define the matrix
$$
    \M(x,y) = \M^{A,B}_k (x,y) = A(y) B(y+x_1) \dots B(y+x_k) \,,
$$
where $x=(x_1,\dots,x_k) \in X$, $y\in Y$. If $y\in Y$ is fixed then
$x=(x_1,\dots,x_k)$ runs over $B^k - \Delta(y)$, i.e. over the set
of cardinality $|B|^k$. If $x=(x_1,\dots,x_k) \in X$ is fixed then
$y$ belongs to the set $A \cap (B-x_1) \cap \dots \cap (B-x_k)$.
Denote by $\l_j = \l_j (A,B,k)$, $j\in [|A|]$ the singular values of
the matrix $\M$. By Corollary \ref{cor:singular_decomposition}, we
have
\begin{equation}\label{f:sum_of_squares}
    \sum_{j=1}^{|A|} \l^2_j = |A| |B|^{k} \,,
\end{equation}
and
\begin{equation}\label{f:sum_of_quater}
    \sum_{j=1}^{|A|} \l^4_j = \E_{2k+1} (A,B)
\end{equation}
because of
$$
    \| \M \|^4_{\Box} = \sum_{y,y'} A(y) A(y') \sum_{x_1,\dots,x_k} \sum_{x'_1,\dots,x'_k}
        B(y+x_1) \dots B(y+x_k) \cdot B(y+x'_1) \dots B(y+x'_k)
            \m
$$
$$
            \m
                B(y'+x_1) \dots B(y'+x_k) \cdot B(y'+x'_1) \dots B(y'+x'_k)
            =
                 \E_{2k+1} (A,B) \,.
$$

We make use of some operators, which were introduced in \cite{s}.

\begin{definition}
Let $\_phi,\psi$ be two complex functions.
By $\oT^{\_phi}_\psi$ denote the following operator on the space of functions $\Gr^{\mathbb{C}}$
\begin{equation}\label{F:T}
    (\oT^{\_phi}_\psi f ) (x) = \psi(x) (\FF{\_phi^c} * f) (x) \,,
\end{equation}
where $f$ is an arbitrary complex function on $\Gr$.
\end{definition}

Let $E\subseteq \Gr$ be a set.
Denote by $\ov{\oT}^{\_phi}_E$ the restriction of operator $\oT^{\_phi}_E$ onto the space of the functions
with supports on $E$.
It was shown in \cite{s}, in particular,
that operators $\oT^{\_phi}_E$ and $\ov{\oT}^{\_phi}_E$ have the same non--zero eigenvalues.
If $\_phi$ is a real function then the operator $\ov{\oT}^{\_phi}_{E}$ is symmetric.
If $\_phi$ is a nonnegative function then the operator is nonnegative definite.
The action of $\ov{\oT}^{\_phi}_{E}$ can be written as
\begin{equation}\label{F:T_S_action}
    \langle \ov{\oT}^{\_phi}_E u, v \rangle
        =
            \sum_x (\FF{\_phi^c} * u) (x) \ov{v} (x)
                =
                    \sum_x \_phi (x) \FF{u} (x) \ov{\FF{v} (x)} \,,
\end{equation}
where $u,v$
are
arbitrary
functions such that $\supp u, \supp v \subseteq E$.
Further properties of such  operators  can be found in \cite{s}.

Using Lemma \ref{l:singular_decomposition} and
the definitions above
we can
give
another
characterization of singular values $\l_j$. We express this in the
next proposition.

\begin{proposition}
    We have
\begin{eqnarray}\label{f:eighenvalues_M}
        \l^2_1 &=&\max_{\| w \|_2 = 1,\, \supp w \subseteq A} \sum_s (w\circ w) (s)(B\circ B) (s)^k  \,, \nonumber\\
                \l^2_2 &=& \max_{\| w \|_2 = 1,\, \supp w \subseteq A,\,\atop w\perp w_1} \sum_s (w\circ w) (s)(B\circ B)(s)^k \nonumber \,,\\
& &\dots  \\
\l^2_{|A|} &=& \max_{\| w \|_2 = 1,\, \supp w \subseteq A,\,\atop w\perp w_1,\, \dots,\, w\perp w_{|A|-1}}
                                            \sum_s (w\circ w) (s)(B\circ B)(s)^k \,,\nonumber
    \end{eqnarray}
    where $w_1,\dots,w_{|A|}$ are eigenvectors of $\M^* \M$.
    In particular, $\l^2_j (A,B,k) = \l^2_j (\pm A, \pm B,k)$, $j\in [|A|]$.
    Furthermore, if $\Gr$ is a finite group, then
    $\l^2_j$ coincide with eigenvalues of the operator $\ov{\oT}^{\_phi}_A$ with
    $\_phi (x) = \frac{1}{|\Gr|} ((B\circ B)^k)\FF ~ (x)$.
\label{p:la^2_j}
\end{proposition}
\begin{proof}
For  $x=(x_1,\dots,x_k) \in B^k$ and an arbitrary  function $w$,
$\supp w \subseteq A$, we have
\begin{eqnarray*}
    \| \M w \|^2_2 &=& \sum_{x_1,\dots,x_k} \Big| \sum_y \M(x,y) w(y) \Big|^2\\
    &=&\sum_{x_1,\dots,x_k} \sum_{y,y'} w(y) w(y') B(y+x_1) \dots B(y+x_k) \cdot B(y'+x_1) \dots B(y'+x_k)\\
&=&  \sum_s (w\circ w) (s)(B\circ B)(s)^k \,,
\end{eqnarray*}
which gives (\ref{f:eighenvalues_M}). Further, by the obtained
formula and the fact $(C^c\circ C^c) (x)^k = (C\circ C) (x)^k$ for
any set $C\subseteq \Gr$, we get $\l^2_j (A,B,k) = \l^2_j (\pm A,
\pm B,k)$, $j\in [|A|]$.
Finally,
the last assertion easily follows from (\ref{F:T_S_action}).
$\hfill\Box$
\end{proof}

\bigskip

Thus, taking $w(x) = A(x) /|A|^{1/2}$, we obtain
\begin{equation}\label{f:la^2_1_low}
    \l^2_1 \ge \frac{\E_{k+1} (A,B)}{|A|} \,.
\end{equation}
Note also the function $\_phi$ above satisfies $\_phi^c (x) = \_phi (x)$
and the following holds
$((B\circ B)^k)^c (x) = (B\circ B)^k (x)$.

We are in position to prove a  lower bound for $R^{(k)}_B [A]$
and even for more general quantities (see estimate (\ref{f:parameter_summation}))
in terms of the energies $\E_{2k+1} (A,B)$.

\begin{theorem}
    Let $A,B \subseteq \Gr$ be sets, and $k\ge 1$ be a positive integer.
    Then
    \begin{equation}\label{f:R&eighenvalues1}
        R^{(k)}_B [A] \ge \frac{|B|^{2k}}{\l^2_1 (A,B,k)} \,,
    \end{equation}
    and
    \begin{equation}\label{f:R&eighenvalues2}
        R^{(k)}_B [A] \ge \frac{|B|^{2k}}{\E_{2k+1}^{1/2} (A,B)} \,.
    \end{equation}
        Moreover, suppose that $A_1 \subseteq A$ is a set
    and $B^{(y)} \subseteq B^k$, $y\in A_1$ is an arbitrary family of sets.
    Then
    \begin{equation}\label{f:parameter_summation}
        \Big| \bigcup_{y\in A_1} (B^{(y)} \pm \Delta(y) ) \Big|
            \ge
                 \frac{\big( \sum_{y\in A_1} |B^{(y)}| \big)^2}{\E_{k+1}(A,B)} \,.
    \end{equation}
\label{t:R&eighenvalues}
\end{theorem}
\begin{proof}
By the definition of the matrix $\M$, we see that for every nonempty
$Z\subseteq A$ we have
\begin{equation}\label{tmp:30.05.2011_1}
    \sigma := \langle \M\, Z, B^k - \Delta(Z) \rangle = \sum_{x,y} \M(x,y) Z(y) (B^k - \Delta(Z)) (x) = |Z| |B|^k \,.
\end{equation}
Using the extremal property of $\l_1$, we get
$$
    \sigma \le \l_1 |Z|^{1/2} |B^k - \Delta(Z)|^{1/2} \,.
$$
Thus
$$
    \frac{|B|^{2k}}{\l^2_1} \le \frac{|B^k - \Delta(X)|}{|X|} = R^{(k)}_B [-A] \,,
$$
where  $X\sbeq A$ is a set that  achieves the minimum in
(\ref{f:R_B[A]}). By Proposition \ref{p:la^2_j},  $\l^2_j (A,B,k) =
\l^2_j (\pm A, \pm B,k)$, for all $j\in [|A|]$, which implies
(\ref{f:R&eighenvalues1}). Finally, (\ref{f:R&eighenvalues2})
follows from (\ref{f:sum_of_quater}).

To prove (\ref{f:parameter_summation}) it is enough to notice that
by (\ref{f:energy-B^k-Delta})
$$\Big| \bigcup_{y\in A_1} (B^{(y)} \pm \Delta(y) ) \Big|
            \ge\frac{\big( \sum_{y\in A_1} |B^{(y)}|
            \big)^2}{\E(\Delta (A),B^k)}=
                 \frac{\big( \sum_{y\in A_1} |B^{(y)}| \big)^2}{\E_{k+1}(A,B)} \,.
$$
This completes the proof.
$\hfill\Box$
\end{proof}

\bigskip
Observe that from
(\ref{f:R&eighenvalues2})
it follows that for
every $A_1, A_2\sbeq A$ we have
$$|A_1\pm A_2|\ge \frac{|A_2|^{2/k} |A_1|^{2}}{\E^{1/k}_{k+1} (A_1,A_2)}
\ge \frac{|A_2|^{2/k} |A_1|^{2}}{\E^{1/k}_{k+1} (A)}\,,$$ for every
$k\ge 1.$

 The theorem above implies some results for sets
with small higher energy. For example,  multiplicative subgroups
$\Z/p\Z$, convex subsets of $\R$ (i.e. sets $A=\{a_1,\dots, a_n\}_<$
such that $a_i-a_{i-1}<a_{i+1}-a_i$ for every $2\le i\le n-1.$)
Another examples are provided by subsets of $\R$ with small product
sets (see section \ref{sec:product_E_k}). We consider here just the
case of multiplicative subgroups.

\begin{corollary}
    Let $p$ be a prime number. Suppose that $\G$ is a multiplicative subgroup of $\Z/p\Z$ with $|\Gamma| = O(p^{2/3})$.
    Then for every set $\Gamma' \subseteq \Gamma$, we have
    \begin{equation}\label{f:subgroups_pred}
        |\Gamma+\Gamma'| \gg |\Gamma'| \cdot \left( \frac{|\Gamma|}{\log |\Gamma|} \right)^{1/2} \,.
    \end{equation}
    Furthermore, for every $k\ge 2$, we get
    $$
        R^{(k)}_\Gamma [\Gamma] \gg |\Gamma|^{k-1/2} \,.
    $$
\label{c:subgroups_pred}
\end{corollary}
\begin{proof}
Indeed, by Lemma 3.3 in \cite{ss}, we have $\E_3 (\Gamma) =
O(|\Gamma|^3 \log |\Gamma|)$, and  $\E_l (\Gamma) = O(|\Gamma|^l)$,
for $l\ge 4$.
Now the assertion follows directly from (\ref{f:R&eighenvalues2}).
$\hfill\Box$
\end{proof}

\bigskip


We show that  the bounds in Corollary \ref{c:subgroups_pred}
can be
improved  for multiplicative subgroups. It turns out that in
this case we know all singular values $\l^2_j$ as well as all
eigenfunctions.

Let $p$ be a prime number, $q=p^s$ for some integer $s \ge 1$. Let
$\F_q$ be the field with  $q$ elements, and let $\Gamma\subseteq
\F_q$ be a multiplicative subgroup.
Denote by $t$ the cardinality of $\Gamma$, and put $n=(q-1)/t$. Let
also $g$ be a primitive root, then $\Gamma = \{ g^{nl}
\}_{l=0,1,\dots,t-1}$. Let $\chi_\a (x)$, $\a \in [t]$ be the
orthogonal family of multiplicative characters on $\Gamma$, that is
$$
    \chi_\a (x) = \Gamma(x) e\left( \frac{\a l}{t} \right) \,, \quad x=g^{nl} \,, \quad 0\le l < t \,.
$$

\begin{proposition}
    Let
    $\Gamma\subseteq \F_q$ be a multiplicative subgroup,
    and let
    $\_phi$ be a $\Gamma$--invariant function.
    Then the functions $\chi_\a (x)$ are eigenfunctions of the operator $\ov{\oT}^{\_phi}_{\Gamma}$.
    If
    $\_phi$ has non--negative Fourier transform
    then
    $\E_{k+1} (\Gamma, \h{\_phi^c})/(|\Gamma| q)$ is the maximal eigenvalue
    corresponding with the eigenfunction $\Gamma(x)$.
    Furthermore, for any $\Gamma$--invariant function $\psi$, $\psi(x)=\psi(-x)$,
    $\h{\psi} (x) \ge 0$ and an arbitrary real function $u$ with support on $\Gamma$, we have
    \begin{equation}\label{f:connected}
        \sum_{x} \psi (x) (u \circ u) (x)
            \ge
                |\Gamma|^{-2} \Big| \sum_{x\in \Gamma} u(x) \Big|^2
                        \cdot
                    \sum_{x} \psi (x) (\Gamma \circ \Gamma) (x) \,.
   \end{equation}
\label{p:eigenfunctions_Gamma}
\end{proposition}
\begin{proof}
We have to show that
$$
    \mu f(x) = \Gamma(x) (\FF{\_phi^c} * f) (x) \,, \quad \mu \in \R
$$
for $f(x) = \chi_\a (x)$. By the assumption $\_phi (x)$ is a
$\Gamma$--invariant function, whence so is $\FF{\_phi^c}$. Thus, for
every $\gamma\in \Gamma$, we have
\begin{eqnarray*}
    (\FF{\_phi^c} * f) (\gamma) &=& \sum_z f(z) \FF{\_phi^c} (\gamma -z) = \sum_z f(\gamma z) \FF{\_phi^c} (\gamma -\gamma
    z)\\
    &=&
    f(\gamma) \cdot \sum_z f(z) \FF{\_phi^c} (1-z) = f(\gamma) \cdot (\FF{\_phi^c} * f) (1) \,.
\end{eqnarray*}
Further,
for every $\a \in [|\Gamma|],\,$ $\E_{k+1} (\Gamma, \h{\_phi^c}) \ge
\E_{k+1} (\chi_\a, \h{\_phi^c})$.

Next, we prove (\ref{f:connected}). Let $\_phi$ be such that $\psi =
\h{\_phi^c}$. Since $\psi(x)=\psi(-x)$, it follows that $\_phi$ is a
real function and, consequently, the operator
$\ov{\oT}^{\_phi}_{\Gamma}$ is symmetric. By assumption $\h{\psi}
(x) \ge 0$, so $\ov{\oT}^{\_phi}_{\Gamma}$ is nonnegative definite
and all its eigenvalues $\mu_\a (\ov{\oT}^{\_phi}_{\Gamma})$ are
nonnegative. If $u=\sum_{\a} c_\a \chi_\a$ then
$$
    \sum_{x} \psi (x) (u \circ u) (x)
        =
            \langle \ov{\oT}^{\_phi}_{\Gamma} u, u \rangle
                =
                    \sum_\a |c_\a|^2 |\Gamma| \mu_\a (\ov{\oT}^{\_phi}_{\Gamma})
                        \ge
                            |\Gamma|^{-2} \langle u, \Gamma \rangle^2
                                \sum_{x} \psi (x) (\Gamma \circ \Gamma) (x)
$$
and the result follows.
$\hfill\Box$
\end{proof}

\bigskip

In particular, we have equality in (\ref{f:la^2_1_low}) for
multiplicative subgroups. Note also that an analog of the
proposition above holds for an arbitrary tiling not necessary for
tiling by cosets.

\begin{corollary}
    Let
    $\Gamma_* \subseteq \F_q$
    be a coset of a multiplicative subgroup $\Gamma$.
    Then for every set $\Gamma' \subseteq \Gamma_*$, and every $\Gamma$--invariant set $Q$, we have
    \begin{equation}\label{f:subgroups_pred1}
        |Q+\Gamma'| \ge |\Gamma'| \cdot \frac{|\Gamma| |Q|^2}{\E_2 (\Gamma_*,Q)} \,.
    \end{equation}
    If $Q^{(y)} \subseteq Q^k,\, y\in \Gamma',$ is an arbitrary family
    of sets, then
    $$
        \Big| \bigcup_{y\in \Gamma'} (Q^{(y)} \pm \Delta (y)) \Big|
            \ge \frac{|\Gamma|}{|\Gamma'| \E_{k+1} (\Gamma_*,Q)} \cdot \Big( \sum_{y\in \Gamma'} |Q^{(y)}| \Big)^2 \,.
    $$
    Furthermore, for each $k\ge 2$, we have
    \begin{equation}\label{tmp:31.05.2011_1}
        R^{(k)}_Q [\Gamma_*] \ge \frac{|\Gamma| |Q|^{2k}}{\E_{k+1} (\Gamma_*,Q)} \,.
    \end{equation}
\label{c:subgroups}
\end{corollary}
\begin{proof}
For every $\xi \in \F^*_q / \Gamma$ and $\a\in [|\Gamma|],$ let us
define the functions $\chi^{\xi}_\a (x) := \chi_\a (\xi^{-1} x).$
Then, clearly   $\supp \chi^{\xi}_\a=\xi \cdot \Gamma$ and
$\chi^{\xi}_\a (\gamma x) = \chi_\a (\gamma) \chi^{\xi}_\a (x)$ for all
$\gamma \in \Gamma$. Using the argument from  Proposition
\ref{p:eigenfunctions_Gamma} it is easy to see that the functions
$\chi^{\xi}_\a$ are orthogonal eigenfunctions of the operator
$\ov{\oT}^{\_phi}_{\Gamma_*}$. This completes the proof.
$\hfill\Box$
\end{proof}

\bigskip

It is easy to see that in the case $\F_q = \Z/p\Z$, $p$ is a prime
number, $|\Gamma| = O(p^{2/3})$ the bound (\ref{tmp:31.05.2011_1})
is  best possible up to a constant factor.
In particular, it gives asymptotic formulas for the sizes of the sets $\G^k \pm \Delta(\G)$, $k\ge 3$.

To apply the  inequality (\ref{f:connected}) of Proposition
\ref{p:eigenfunctions_Gamma} we need  a lemma (see, e.g. \cite{sv} or \cite{K_Tula,KS1}).

\begin{lemma}
{
    Let $p$ be a prime number,
    $\Gamma\subseteq \F^*_p$ be a multiplicative subgroup, and
    $Q,Q_1,Q_2\subseteq \F^*_p$ be any $\Gamma$--invariant sets such that
    $|Q| |Q_1| |Q_2| \ll |\Gamma|^{5}$ and $|Q| |Q_1| |Q_2| |\Gamma| \ll p^3$.
    Then
        \begin{equation}\label{f:improved_Konyagin_old3}
            \sum_{x\in Q} (Q_1 \circ Q_2) (x) \ll |\Gamma|^{-1/3} (|Q||Q_1||Q_2|)^{2/3} \,.
        \end{equation}
}
\label{l:improved_Konyagin_old}
\end{lemma}

Using Lemma \ref{l:improved_Konyagin_old}, one can easily deduce
bounds for moments of convolution of $\Gamma$, e.g. (see \cite{ss})
that $\E(\Gamma) = O(|\Gamma|^{5/2})$ and $\E_3 (\Gamma) =
O(|\Gamma|^3 \log |\Gamma|)$, provided that $|\Gamma| = O(p^{2/3})$.

\begin{corollary}
    Let $p$ be a prime number, and $\Gamma \subseteq \F^*_p$ be a multiplicative subgroup,
    $|\Gamma| = O(p^{1/2})$.
    Then
    \begin{equation}\label{f:subgroups_energy&sumset}
        \E (\Gamma) \ll |\Gamma|^{\frac{23}{12}} |\Gamma \pm \Gamma|^{\frac{1}{3}} \log^{\frac12} |\Gamma| \,.
    \end{equation}
    and
    \begin{equation}\label{f:subgroups_energy&sumset_2}
        \E (\Gamma) \ll |\Gamma|^{\frac{31}{18}} |\Gamma \pm \Gamma|^{\frac{4}{9}} \log^{\frac12} |\Gamma| \,.
    \end{equation}
\label{c:subgroups_energy&sumset}
\end{corollary}
\begin{proof}
We have $\E(\Gamma) = O(|\Gamma|^{5/2})$.
One can assume that
\begin{equation}\label{f:new_Q}
    |\Gamma \pm \Gamma| = O\left( \frac{\E^3(\Gamma)}{|\Gamma|^{23/4} \log^{1/2} |\Gamma|} \right)
        = O \left( \frac{|\Gamma|^{7/4}}{\log^{1/2} |\Gamma|} \right)
\end{equation}
because otherwise inequality (\ref{f:subgroups_energy&sumset}) is
trivial. To obtain (\ref{f:subgroups_energy&sumset}), we use a
formula from  \cite{Li} (see Lemma 2.5)
$$
    \Big( \sum_{x} (\Gamma \circ \Gamma)^{3/2} (x) \Big)^2 |\Gamma|^2
        \le
            \E_3 (\Gamma) \cdot \E (\Gamma, \Gamma \pm \Gamma) \,.
$$
Further, by the assumption $|\Gamma| = O(p^{3/4})$ and Lemma
\ref{l:improved_Konyagin_old}, we have (see also the proof of
Theorem 1.1 from \cite{Li})
$$
    \E^3 (\Gamma) \ll |\Gamma|^3 \cdot \Big( \sum_{x} (\Gamma \circ \Gamma)^{3/2} (x) \Big)^2 \,.
$$
Combining the last two formulas, we obtain
\begin{equation}\label{tmp:31.08.2011_2}
    \E^3 (\Gamma) \ll |\Gamma| \E_3 (\Gamma) \cdot \E (\Gamma, \Gamma \pm \Gamma) \,.
\end{equation}
We show that for every  $\Gamma$--invariant set $Q$ we have
\begin{equation}\label{tmp:31.08.2011_1}
    \sum_x (Q \circ Q) (x) (\Gamma \circ \Gamma)^2 (x) \ge |\Gamma|^{-2} \E(\Gamma) \cdot \E (\Gamma,Q) \,.
\end{equation}
By (\ref{f:connected}) of Proposition \ref{p:eigenfunctions_Gamma}
with $\psi(x) = (Q \circ Q) (x)$ and $u(x) = \Gamma_s (x) = (\Gamma
\cap (\Gamma-s)) (x)$, we get
\begin{equation}\label{tmp:07.09.2011_3}
    \sum_x (Q \circ Q) (x) (\Gamma_s \circ \Gamma_s) (x) \ge \frac{|\Gamma_s|^2}{|\Gamma|^{2}} \E (Q,\Gamma) \,.
\end{equation}
Summing over $s\in \Gamma-\Gamma$, we obtain
(\ref{tmp:31.08.2011_1}). Inserting  (\ref{tmp:31.08.2011_1}) in
(\ref{tmp:31.08.2011_2}),  we infer that
$$
    \E^4 (\Gamma) \ll |\Gamma|^3 \E_3 (\Gamma) \cdot \sum_x (Q \circ Q) (x) (\Gamma \circ \Gamma)^2 (x) \,,
$$
where $Q=\Gamma \pm \Gamma$. By the assumption $|\Gamma| =
O(p^{1/2})$. Let us prove that
\begin{equation}\label{tmp:07.09.2011_1}
    \sum_x (Q \circ Q) (x) (\Gamma \circ \Gamma)^2 (x) \ll \frac{|Q|^{4/3}}{|\Gamma|^{2/3}} |\Gamma|^{7/3} \log |\Gamma|
        \ll
            |Q|^{4/3} |\Gamma|^{5/3} \log |\Gamma| \,.
\end{equation}
From (\ref{tmp:31.08.2011_1}) and (\ref{tmp:31.08.2011_2}) it
follows that the summation in the (\ref{tmp:07.09.2011_1}) can be
taken over $x$ such that
\begin{equation}\label{f:H}
    (Q \circ Q) (x) \ge \frac{\E(\Gamma,Q)}{2|\Gamma|^2} \gg \frac{\E^3(\Gamma)}{|\Gamma|^3 \E_3(\Gamma)} := H \,.
\end{equation}
Hence, it is sufficient to prove that
\begin{equation}\label{tmp:07.09.2011_2}
    \sum_{x ~:~ (Q \circ Q) (x) \ge H } (Q \circ Q) (x) (\Gamma \circ \Gamma)^2 (x)
        \ll
            |Q|^{4/3} |\Gamma|^{5/3} \log |\Gamma| \,.
\end{equation}
Let $(Q \circ Q) (\xi_1) \ge (Q \circ Q) (\xi_2) \ge \dots$ and
$(\Gamma \circ \Gamma) (\eta_1) \ge (\Gamma \circ \Gamma) (\eta_2)
\ge \dots$, where $\xi_1,\xi_2, \dots$ and $\eta_1,\eta_2, \dots$
belong to distinct cosets. Applying Lemma
\ref{l:improved_Konyagin_old} once more, we get
\begin{equation}\label{tmp:07.09.2011_2}
    (Q \circ Q) (\xi_j) \ll \frac{|Q|^{4/3}}{|\Gamma|^{2/3}} j^{-1/3} \,,
        \quad \mbox{ and } \quad
            (\Gamma \circ \Gamma) (\eta_j) \ll |\Gamma|^{2/3} j^{-1/3} \,,
\end{equation}
provided that  $j|\Gamma| |Q|^2 \ll |\Gamma|^5$ and $j|\Gamma| |Q|^2
|\Gamma| \ll p^3$. We have $j\ll |Q|^4/ (|\Gamma|^2 H^3)$,
$\E_3 (\Gamma) = O(|\Gamma|^3 \log |\Gamma|)$ and $|\Gamma| =
O(p^{1/2})$, thus, the last conditions are satisfied. Applying
(\ref{tmp:07.09.2011_2}), we obtain (\ref{tmp:07.09.2011_1}). Using
the fact $\E_3 (\Gamma) = O(|\Gamma|^3 \log |\Gamma|)$, and the
formula (\ref{tmp:07.09.2011_1}), we
get
$$
    \E^4 (\Gamma) \ll |\Gamma|^6 \log |\Gamma| \cdot |Q|^{4/3} |\Gamma|^{5/3} \log |\Gamma|
$$
and
(\ref{f:subgroups_energy&sumset})
is proved.

To show (\ref{f:subgroups_energy&sumset_2}), we just put $u(x) =
(\Gamma \cap (Q-s)) (x)$ in (\ref{tmp:07.09.2011_3}) instead of
$u(x) = \Gamma_s (x)$. We have
\begin{equation}\label{tmp:29.09.2011_1}
    \sum_x (Q \circ Q)^2 (x) (\Gamma \circ \Gamma) (x) \ge |\Gamma|^{-2} \cdot \E^2 (\Gamma,Q) \,,
\end{equation}
where $Q=\Gamma\pm \Gamma$.
Applying (\ref{tmp:31.08.2011_2}), we get
\begin{equation}\label{tmp:31.08.2011_4}
    \sum_x (Q \circ Q)^2 (x) (\Gamma \circ \Gamma) (x) \cdot |\Gamma|^4 \E^2_3 (\Gamma) \ge \E^6 (\Gamma) \,.
\end{equation}
As before, we need an analog of the estimate
(\ref{tmp:07.09.2011_1})
\begin{equation}\label{tmp:31.08.2011_5}
    \sum_x (Q \circ Q)^2 (x) (\Gamma \circ \Gamma) (x)
        \ll
              |Q|^{8/3} |\Gamma|^{1/3} \log |\Gamma| \,.
\end{equation}
Again, using the inequality (\ref{tmp:29.09.2011_1}) and the
definition of  $H$  (\ref{f:H}), it is sufficient to prove that
$$
    \sum_{ x ~:~ (Q \circ Q) (x) \ge H} (Q \circ Q)^2 (x) (\Gamma \circ \Gamma) (x)
        \ll
              |Q|^{8/3} |\Gamma|^{1/3} \log |\Gamma| \,.
$$
One can assume that an analog of (\ref{f:new_Q}) holds
\begin{equation}\label{f:new_Q'}
     |Q| = O\left( \frac{\E^{9/4}(\Gamma)}{|\Gamma|^{31/8} \log^{1/2} |\Gamma|} \right)
        = O \left( \frac{|\Gamma|^{7/4}}{\log^{1/2} |\Gamma|} \right)
\end{equation}
because otherwise the inequality (\ref{f:subgroups_energy&sumset_2})
is trivial. Using previous arguments, the bound (\ref{f:new_Q'}) and
applying Lemma \ref{l:improved_Konyagin_old}, and inequalities
$|\Gamma| \ll p^{1/2}$,
$\E_3 (\Gamma) \ll |\Gamma|^3 \log |\Gamma|$,
we
get the required estimate. Inserting (\ref{tmp:31.08.2011_5}) in
(\ref{tmp:31.08.2011_4}), and using $\E_3 (\Gamma) \ll |\Gamma|^3
\log |\Gamma|$ once again, we obtain
(\ref{f:subgroups_energy&sumset_2}). $\hfill\Box$
\end{proof}

\bigskip

In particular, if $\E (\Gamma) \gg |\Gamma|^{5/2}$ then $|\Gamma \pm
\Gamma| \gg |\Gamma|^{\frac{7}{4} - \epsilon}$, for any
$\epsilon>0$. At the moment it is known (see \cite{sv}),
unconditionally, that $|\Gamma - \Gamma| \gg |\Gamma|^{\frac{5}{3} -
\epsilon}$, for an arbitrary $\epsilon>0$ and any multiplicative
subgroup $\Gamma$ with  $|\Gamma| = O(p^{1/2})$. Note also that the
condition $|\Gamma| = O(p^{1/2})$ in the previous result can be
slightly relaxed.

\begin{corollary}
    Let $\Gamma \subseteq \F_p^*$ be a multiplicative subgroup such that $-1\in \G$,
    $|\Gamma| \ge p^{\kappa}$, where
     $\kappa > \frac{99}{203}.$
    Then for all sufficiently large $p$ we have $\F^*_p \subseteq 6\Gamma$.
\label{c:6R_subgroups}
\end{corollary}
\begin{proof}
Put $S = \Gamma+\Gamma$, $n=|\Gamma|$, $m = |S|$, and $\rho =
\max_{\xi \neq 0} |\h{\Gamma} (\xi)|$. By a well--known upper bound
for Fourier coefficients of multiplicative subgroups (see e.g.
Corollary 2.5 from \cite{ss}) we have $\rho \le p^{1/8} \E^{1/4}
(\Gamma)$. If $\F_p^* \not\subseteq 6\Gamma$ then for some
$\lambda\neq 0,$ we
obtain
$$
    0 = \sum_{\xi} \h{S}^2 (\xi) \h{\Gamma}^2 (\xi) \h{\l \Gamma} (\xi)
        =
            m^2n^3+\sum_{\xi\neq 0} \h{S}^2 (\xi) \h{\Gamma}^2 (\xi) \h{\l \Gamma} (\xi) \,.
$$
Therefore, by the estimate $\rho \le p^{1/8} \E^{1/4} (\Gamma)$
and Parseval identity we get
$$n^3m^2\le \rho^3mp\ll (p^{1/8} \E^{1/4})^3 mp \,.$$
Now applying formula (\ref{f:subgroups_energy&sumset_2})
and
$m \gg n^{5/3} \log^{-1/2} n$ (see \cite{sv}),
we obtain the required result.
$\hfill\Box$
\end{proof}

\bigskip

The inclusion $\F^*_p \subseteq 6\Gamma$ was obtained in \cite{sv} under the
assumption $\kappa > \frac{33}{67}$.

\section{Two versions of Balog--Szemer\'edi--Gowers theorem}
\label{sec:BSzG}

We show here two  versions of the  Balog--Szemer\'edi--Gowers
theorem (see Lemma \ref{l:bsg}) in the case when $\E_k(A)$ is not
much bigger than the trivial lower bound in terms of additive energy
i.e. $\E(A)^{k-1}/|A|^{2k-4}.$ The first result (Theorem
\ref{t:BSzG1}) provides an improvement on the size of a "structured"
subset $A'$ of $A$ and the size of $A'-A',$ as well, assuming that
$\E_{2+\e}(A)$ is "small". Our method essentially follows, with some
modifications, the Gowers proof \cite{gowers}. Our second theorem
(Theorem \ref{t:E_BSzG2}) gives a near optimal estimate on
$|A'-A'|$, again for a very large $A'\sbeq A,$ however we have to
assume that $\E_{3+\e}(A)$ is small. To prove Theorem
\ref{t:E_BSzG2} we develop the idea used in the proof of Theorem
\ref{t:small E_3}. At the end of this section we establish some
results concerning sumsets and energies of multiplicative subgroups
and convex sets.

 We will need two lemmas. The first one is  a version of
Gowers Lemma 7.4 \cite{gowers}, see also Lemma 1.9 in
\cite{sanders:pop}.

\begin{lemma}\label{l:intersection}
Let $I$ and $S$ be sets with $|I|=n$ and $|S|=m.$ Suppose that
$S_i\sbeq S,\, i\in I,$ is a family of sets such that
$$\sum_{i,j\in I}|S_i\cap S_j|\gs \d^2 mn^2\,,$$
where $0\ls \d\ls 1.$ Let $\eta>0.$ Then there is $J\sbeq I\,,
|J|\gs \d n/\sqrt{2}$ such that
\begin{equation}\label{e:J}\big |\big \{(i,j)\in J\times J\, :\, |S_i\cap S_j|\gs \eta \d^2 n/2\big \}\big |\gs (1-\eta)|J|^2\,.
\end{equation}
\end{lemma}
\begin{proof} We have
\begin{equation}\label{e:everage}
\d^2 mn^2\ls \sum_{i,j\in I}|S_i\cap S_j|=\sum_{\a}\sum_{i,j\in
I}S_i(\a)S_j(\a)\,.
\end{equation}
For $\a\in S,$ we put $K_\a= \{ i\in I~:~ \a\in S_i\}.$ Clearly
$K_\a(i)=S_i(\a)$, so we can rewrite (\ref{e:everage}) as
$$\d^2 mn^2\ls \sum_\a|K_\a|^2\,.$$
Let
$$Y=\big \{(i,j)\in I\times I\,:\, |S_i\cap S_j|< \eta \d^2 m/2\big \}\,,$$
then
$$\sum_{\a\in S}|K_\a\times K_\a \cap Y|=\sum_{(i,j)\in Y}|S_i\cap S_j|< \eta \d^2 mn^2/2\,,$$
so that
$$\sum_{\a\in S} |K_\a|^2-\eta^{-1}\sum_{\a\in S}|K_\a\times K_\a \cap Y|>\d^2 mn^2/2\,.$$
Thus, there exists $\a\in S$ such that $|K_\a|\gs \d n/\sqrt 2$ and
$|K_\a\times K_\a \cap Y|<\eta |K_\a|^2.$ It is enough to observe
that the assertion holds with $J=K_\a.\hfill\Box$
\end{proof}

\begin{corollary}\label{c:intesection} With the assumption of Lemma \ref{l:intersection} there is a set $J'\sbeq J$ of size
at least $2^{-5}\d n$ such that for every $i,j\in J'$ there are at
least $2^{-2}\d n$ elements $k\in I$ with
$$|S_i\cap S_k|\gs  2^{-4}\d^2 m,~~~~|S_j\cap S_k|\gs 2^{-4} \d^2 m\,.$$
\end{corollary}
\begin{proof} Applying the previous lemma with $\eta=1/8$, we obtain a set $J$ satisfying (\ref{e:J}). Let $V$ be the
set of all pairs $(i,j)\in J\times J$ such that $|S_i\cap S_j|\gs
2^{-1}\eta \d^2 m=2^{-4}\d^2m.$ Then we have
$$\sum_{(i,j)\in J\times J}V(i,j)\gs (1-\eta)|J|^2=\frac78|J|^2\,.$$
Put $J'=\big \{i\in J\,:\, \sum_jV(i,j)\gs \frac34|J|\big \}.$
Clearly
$$\sum_{i\in J'
}\sum_{j\in J}V(i,j)\gs \frac1{16}|J|^2\,,$$ whence $|J'|\gs
2^{-4}|J|>2^{-5}\d n.$ Furthermore, observe that if $i,j\in J',$
then
$$\sum_{k}V(i,k)V(j,k)\gs |J|/2\,, $$
as required.$\hfill\Box$
\end{proof}

\bigskip
\no Now we are ready to prove the first  main result of this
section.

\begin{theorem}\label{t:BSzG1}
Let $A$ be a subset of an abelian group. Suppose that $
\E(A)=|A|^3/K$ and $\E_{2+\e}(A)=M|A|^{3+\e}/K^{1+\e}\,.$ Then there
exists $A'\sbeq A$ such that $|A'|\gg |A|/(2M)^{1/\e}$ and
$$|A'-A'|\ll 2^{\frac{6}{\e}} M^{\frac{6}{\e}} K^4|A'|\,.$$
\end{theorem}
\begin{proof} Observe that
$$\frac{|A|^3}{K}=\E(A)=\sum_a A(a)\sum_bA(b)(A\c A)(a-b)\,.$$
For $a\in A,$ we set
$$S_a=\big \{b\in A\,:\, (A\circ A)(a-b)\gs |A|/(2K)\big \}\,,$$
hence
\begin{equation*}\label{e:pop}
\sum_a A(a)\sum_bS_a(b) (A\circ A)(a-b)\ge \frac{|A|^3}{2K}\,.
\end{equation*}
By H\"older  inequality we have
\begin{eqnarray}\label{ineq:e_3}
\sum_{a} A(a)\sum_b S_a(b) (A\circ A)(a-b) &\ls&
 \sum_{a}A(a) \Big(\sum_b S_a(b) \Big)^{\frac\e{1+\e}}\Big(\sum_b  S_a(b)(A\circ A)(a-b)^{1+\e} \Big)^{\frac1{1+\e}}\nonumber \\
 &\ls& \Big(\sum_a |S_a|\Big)^{\frac\e{1+\e}} \Big (\sum_aA(a) \sum_bA(b)
(A\circ A)(a-b)^{1+\e}\Big )^{\frac1{1+\e}} \nonumber\\
&\le&
   \Big(\sum_a |S_a|\Big)^{\frac\e{1+\e}} \E_{2+\e} (A)^{\frac1{1+\e}}\,,
\end{eqnarray}
so that
\begin{equation}\label{e:e_3}
    \sum_a |S_a| \ge  \frac{|A|^2}{2^{\frac{1+\e}{\e}} M^{\frac{1}{\e}}}\,
\end{equation}
and
$$\sum_{a,a'\in A} |S_a\cap S_{a'}|\ge \frac{|A|^3}{2^{\frac{2+2\e}{\e}} M^{\frac{2}{\e}}} \,.$$
We apply Corollary \ref{c:intesection} with $I=S=A,\, n=m=|A|,$ and
the family $\{S_a\},\, a\in A.$ Set $\d= 2^{-\frac{1+\e}{\e}} M^{-\frac{1}{\e}}.$
By Corollary
\ref{c:intesection} there exists a set $A'\sbeq A,\, |A'|\gs
2^{-5}\d n\gg |A|/2^{\frac{1}{\e}} M^{\frac{1}{\e}}$
such that for every $x,y\in A'$ there
are at least $2^{-2}\d n\gg |A|/2^{\frac{1}{\e}} M^{\frac{1}{\e}}$ elements $z\in A$ with
$$|S_x\cap S_z|, ~|S_y\cap S_z|\gs 2^{-4}\d^2 m\,.$$
For each $b\in S_x\cap S_z$ we have
$(A\c A)(x-b),\, (A\c A)(z-b)\gs |A|/(2K).$ Similarly,   for each
$b\in S_y\cap S_z$ we have $(A\c A)(y-b),\, (A\c A)(z-b)\gs
|A|/(2K).$ Therefore,
$$
    ((A\c A)\c (A\c A))(x-z)\gs \sum_{b}(A\c A)(x-b)(A\c A)(z-b)\gs 2^{-4}\d^2 m\frac{|A|^2}{4 K^2} \gg
        \frac{|A|^3}{2^{\frac{2}{\e}} M^{\frac{2}{\e}} K^2}
$$ and the same holds for $y-z.$ Thus,
there are $\gg |A|^7/(2^{\frac{5}{\e}} M^{\frac{5}{\e}} K^4)$ ways to write $x-y$ in the
form $a_1-a_2 +a_3-a_4+a_5-a_6+a_7-a_8,\, a_i\in A.$ Hence
$$|A'-A'|\frac{|A|^7}{2^{\frac{5}{\e}} M^{\frac{5}{\e}} K^4}\ll |A|^8$$
and the assertion follows. $\hfill\Box$
\end{proof}

\bigskip

\begin{corollary}\label{c:bsg-1}
Let $A$ be a subset of an abelian group. Suppose that $
\E(A)=|A|^3/K$ and $\E_3(A)=M|A|^4/K^2\,.$ Then there exists
$A'\sbeq A$ such that $|A'|\gg |A|/M$ and
$$|A'-A'|\ll M^6K^4|A'|\,.$$
\end{corollary}

Now, using a different approach, we prove the following almost
optimal version of Balog--Szemer\'edi--Gowers Theorem, provided that
$\E_{3+\e}(A), \, \e>0,$ is small.

\begin{theorem}\label{t:E_BSzG2} Suppose that $\E(A)= |A|^3/K$ and $\E_{3+\e}(A)= M|A|^{4+\e}/K^{2+\e},$
where $\eps \in (0,1]$.
Then there exists $A'\sbeq A$ such that $|A'|\gg M^{-\frac{3+6\e}{\e(1+\e)}} |A|$ and
$$|nA'-mA'|\ll M^{6(n+m)\frac{3+4\e}{\e(1+\e)}}
K |A'|$$ for every $n,m\in \N.$\end{theorem}
\begin{proof} Let $P$ be the set of popular differences
with at least $|A|/2K$ representations. Similarly,  as in the proof of Theorem \ref{t:small E_3}
we have
$$ \sum_{s\in P} (A\c A) (s)^2\gs \frac12 \E(A) \ge \frac{|A|^3}{2K}$$
and
$$\frac{|A|^3}K\ll \E_{3+\e}(A)^{\frac{2}{3+\e}} |P|^{\frac{1+\e}{3+\e}}\,,$$
so  $|P|\gg K|A|/ M^{2/(1+\e)}.$ Furthermore, by H\"{o}lder inequality
\begin{equation}\label{f:E_3_via_E_k}
 \sum_{s\in P}(A\c A)(s)\gg \frac{\E(A)^{\frac{2+\e}{1+\e}}}{\E_{3+\e}(A)^{\frac{1}{1+\e }}}\gg {M^{-\frac{1}{1+\e}}}|A|^2\,.
\end{equation}

As in Theorem \ref{t:BSzG1}, put $S_a=A\cap (a-P)$, $a\in A$ i.e. $S_a$ is
the set of all $b\in A$ such that $a-b\in P.$ We show that $P$ has
huge additive energy. To do this we apply a generalization of
Katz--Koester transform.
Observe that for every $s\in A-A$ we have
$$\bigcup_{a\in A_s}\big(a-(S_a\cap S_{a-s})\big)\sbeq P\cap (P+s).$$
From  Cauchy-Schwarz inequality  it follows that
\begin{equation}\label{e:main-ineq}
(P\c P)(s)\ge \Big |\bigcup_{a\in A_s}\big(a-(S_a\cap S_{a-s})\big) \Big |
\ge
\frac{(\sum_{a\in A_s} |S_a\cap
S_{a-s}|)^{2}}{\E(A_s,A)}\,.
\end{equation}

By (\ref{f:E_3_via_E_k}), we have
\begin{equation}\label{f:E_3_via_E_k'}
    \sum_{s\in P}(A\c A)(s)=\sum_{a\in A} |S_a|:= \g |A|^2\gg \max(M^{-\frac{1}{1+\e}},K^{-1}|A|^{-1}|P|)\,|A|^2\,,
\end{equation}
so that
$$\sum_{a,a'\in A} |S_a\cap S_{a'}|=\sum_{s}\sum_{a\in A_s}|S_a\cap S_{a-s}|\gg \g^2|A|^3.$$
Let $p>1$ and $q>1$, then by H\"older inequality and Lemma \ref{l:E_k-identity'}
\begin{eqnarray*}\label{e:new-ineq}
\g^2|A|^3&\ll&
            \left( \sum_{s} \frac{(\sum_{a\in A_s} |S_a\cap S_{a-s}|)^{q}}{\E(A_s,A)^{q/2}} \right)^{\frac1q}
                    \left( \sum_{s}  \E(A_s,A)^{\frac{q}{2(q-1)}} \right)^{\frac{q-1}{q}}\\
    &\ll&
                 \E_{\frac{q}{2}} (P)^{\frac1q} \left( \sum_{s} |A_s|^{\frac{(p-1)q}{(q-1)p}} \E_{1+p}(A_s,A)^{\frac{q}{2p(q-1)}} \right)^{\frac{q-1}{q}}\\
    &\le&        \E_{\frac{q}{2}} (P)^{\frac1q} \left( \sum_{s} |A_s|^{\frac{2(p-1)q}{2p(q-1)-q}}\right )^{\frac{2p(q-1)-q}{2pq}}   \E_{2+p}(A)^{\frac{1}{2p}}\\
    &=& \E_{\frac{q}{2}} (P)^{\frac1q} \E_{\frac{2(p-1)q}{2p(q-1)-q}}(A)^{\frac{2p(q-1)-q}{2pq}}   \E_{2+p}(A)^{\frac{1}{2p}}\,.
\end{eqnarray*}
In particular, taking $p=1+\e, \, q=2p$, we get
$$\g^2|A|^3\ll \E_{1+\e} (P)^{\frac1{2+2\e}} |A|^{\frac{2\e}{1+\e}}   \E_{3+\e}(A)^{\frac{1}{2+2\e}}\,,$$
hence
$$\E_{1+\e}(P)\gg (\g^2|A|^{3-\frac{2\e}{1+\e}})^{2+2\e}  \E_{3+\e}(A)^{-1}$$
 In view of the inequality $K|A|/M^{\frac2{1+\e}}\ll |P|\le 2K|A|$ and the definition of $\g,$ we infer that
\begin{eqnarray}\label{e:energy-est}
\E(P)&\ge & |P|^{2-\frac{2}{\e}}\E_{1+\e}(P)^{\frac1\e}
\\
&\gg&  |P|^{2-\frac{2}{\e}}\g^{4+\frac4\e}|A|^{\frac6\e+2} \E_{3+\e}(A)^{-\frac1\e}\\
&=&\frac{\g^4K|A|}{|P|} \left (\frac{\g^{4}K^{2}|A|^{2}}{ M|P|^{2}}\right )^{\frac{1}{\e}}|P|^3\\
&\gg&
 M^{-\b}|P|^3\,,
\end{eqnarray}
where $\b=\frac{3+4\e}{\e(1+\e)}.$
Note that the first inequality in the formula above follows certainly from H\"{o}lder for $\eps\in (0,1)$
but it is also takes place for $\eps = 1$.

Now we proceed as in the proof of  Theorem \ref{t:small E_3}. By
Balog--Szemer\'edi--Gowers Theorem \ref{l:bsg} there exists a set
$P'\sbeq P$ such that $|P'| \gg M^{-\b} |P|,$ and
$$|P'+P'|\ll M^{6\b}  |P'|\,,$$
 so that by Pl\"unnecke--Ruzsa inequality
$$ |nP'-mP'|\ll M^{6(n+m)\b}  |P'| \,.$$
By the pigeonhole principle, we
find $x$ such that
$$|(A-x)\cap P'|\gg |P'|/K \gg |A|/M^{\b+\frac2{1+\e}}.$$
Setting $A'=A\cap (P'+x)$,  the assertion follows.$\hfill\Box$
\end{proof}

\begin{remark}
    In the proof of the theorem above we need the assumption that the energy   $\E_{4} (A)$ is small.
    However, for some sets, for instance  multiplicative subgroups,
    one can apply the inequality
    $\E_{3} (A_s,A) \le \frac{|A_s|}{|A|} \E_{3} (A)$ (see Proposition \ref{p:eigenfunctions_Gamma}) to obtain the
    same result.
\label{r:E_3(A_s,A)}
\end{remark}

\begin{corollary}\label{c:E_4} Suppose that $\E(A)=|A|^3/K$ and $\E_{4}(A)= M|A|^{5}/K^{3}.$
Then there exists $A'\sbeq A$ such that $|A'|\gg M^{-9/2} |A|$ and
$$|nA'-mA'|\ll M^{21(n+m)}
K |A'|$$ for every $n,m\in \N.$
\end{corollary}

\begin{remark}
 Observe that  Corollary \ref{c:E_4}  can be proved by  arguments used in the  proof of Lemma
 \ref{l:sigma_k}. Indeed, let ${\cal G}$ be the popularity graph on $A$ i.e. for $a,b\in A,~$
 $\{a,b\}$ is an edge in ${\cal G}$ if and only if
 $(A\c A)(a-b)\ge |A|/(2K).$
 By (\ref{f:E_3_via_E_k'})
 $$|E({\cal G})|=\sum_{x\in P} (A\c A)(x):=\g |A|^2\gg \max(M^{-1/2}, K^{-1}|A|^{-1}|P|)|A|^2\,.$$
 Let ${\cal C}$ be the family of $\,4-$tuples $(a_1,a_2,a_3,a_4)\in V({\cal G})$
 such that  $\{a_1,a_2\}, \{a_2,a_3\}, \{a_3,a_4\}, \{a_{4},a_1\}\in E({\cal G}).$
Then we have
$$
    |{\cal C}|\ge
    \g^4|A|^4 \,.
$$
 Let us consider a map $\psi : {\cal C} \to P^{4}$ defined in the
following way. If $C=(a_1,a_2,a_3,a_4)\in {\cal C}$  then
$$
    \psi (C) = (a_1-a_2, a_2-a_3, a_3-a_4, a_{4}-a_1) \,.
$$
Arguing as in Lemma \ref{l:sigma_k}, we get
$$\E(P)\ge \frac{|{\cal C}|^2}{\E_4(A)}\gg \g^8 M^{-1}K^3|A|^3$$
The rest of the proof remains the same.
\label{r:E_4-easy}
\end{remark}

As an applications of ideas that appeared Lemma \ref{l:sigma_k} and Theorem \ref{t:E_BSzG2}, we also prove
some estimates on  the size and the additive energy of multiplicative subgroups of $\F_p$ and convex sets.

\begin{corollary}\label{c:f:sigma(P,S)}
    Let $A$ be a convex set.
    Then
    \begin{equation}\label{f:sigma(P,S)_1}
        |A-A| |A|^{\frac{285}{8}} \gg \E(A)^{15} \log^{-\frac{15}{2}} |A| \,.
    \end{equation}
    Further, let $p$ be a prime number, $\G$ be a multiplicative subgroup, $|\G| \ll \sqrt{p}$.
    Then
    \begin{equation}\label{f:sigma(P,S)_2}
        |\G-\G| |\G|^{33} \gg \E(\G)^{14} \log^{-\frac{15}{2}} |\G| \,.
    \end{equation}
\end{corollary}
\begin{proof}
Let $D=A-A$, $\E (A) = |A|^3/K$.
Let also $P$ be the set of popular differences with at least $|A|/2K$ representations.
As before
    $$
        \sum_{a\in A} |S_a| \ge \frac{\E (A)^2}{4\E_3 (A)} \gg |A|^3 K^{-2} \log^{-1} |A| \,,
    $$
    and hence
    $$
        \sum_{a,a'\in A} |S_a \cap S_{a'}| = \sum_{s\in D}\, \sum_{a\in A_s} |S_a \cap S_{a-s}|
            \gg
                |A|^5 K^{-4} \log^{-2} |A| \,.
    $$
    Further, by (\ref{e:main-ineq}), we have
    \begin{eqnarray}\label{tmp:04.12.2011_1}
        |A|^{10} K^{-8} \log^{-4} |A|
            &\ll&
                \left( \sum_{s\in D}\, \sum_{a\in A_s} |S_a \cap S_{a-s}| \right)^2
                    \ll
                        \left( \sum_{s\in D} |P_s|^{1/2} \cdot \E (A_s,A)^{1/2} \right)^2\\
            &\ll&
               \E_3(A) \sum_{s\in D} (P\circ P) (s)  \,.
    \end{eqnarray}
    Whence
    $$
        \sum_{s\in D} (A\c A\c A\c A) (s) \gg \Big (\frac{|A|}{K}\Big )^2\sum_{s\in D} (P\circ P) (s)\gg |A|^9 K^{-10} \log^{-5} |A| \,.
    $$
    By Theorem 1 from \cite{ikrt}, we get
    $$
        |A|^{4-4/3+1/12} |D|^{2/3} \gg |A|^9 K^{-10} \log^{-5} |A|\,,
    $$
    so
    $$
        |D| K^{15} \gg |A|^{\frac{75}{8}} \log^{-\frac{15}{2}} |A|
    $$
    and finally,
    $$
        |D| |A|^{\frac{285}{8}} \gg \E (A)^{15} \log^{-\frac{15}{2}} |A| \,.
    $$

    To get (\ref{f:sigma(P,S)_2}) return to (\ref{tmp:04.12.2011_1}) and obtain
    $$
        |\G|^7 K^{-8} \log^{-5} |\G|
            \ll
                \sum_{s\in D} (P\circ P) (s)  \,.
    $$
    We can suppose that $|P|^2 |D| \ll (K|\G|)^2 |D| \ll |\G|^5$,
    because otherwise
    $|D| \gg |\G|^3 K^{-2}$ and (\ref{f:sigma(P,S)_2}) holds.
    Thus, $|P|^2 |D| |\G| \ll |\G|^6 \ll p^3$ by the assumption $|\G| \ll \sqrt{p}$.
    Applying Lemma \ref{l:improved_Konyagin_old}, we have
    $$
        |\G|^7 K^{-8} \log^{-5} |\G|
            \ll
                \frac{(K|\G|)^{4/3} |D|^{2/3}}{|\G|^{1/3}} \,.
    $$
    In other words
    $$
        |\G|^{9} \ll  |D| K^{14} \log^{\frac{15}{2}} |\G|
            \ll
                |D| |\G|^{42} \E^{-14} (\G) \log^{\frac{15}{2}} |\G|
    $$
    and we obtain (\ref{f:sigma(P,S)_2}).
    This completes the proof. $\hfill\Box$
\end{proof}

\bigskip

In particular, if $\E(A) \sim |A|^{\frac{5}{2}}$ then $|A-A| \gg
|A|^{\frac{15}{8}} \log^{-\frac{15}{2}} |A|$ for any convex set.
Similarly, if $\E(\G) \sim |\G|^{\frac{5}{2}}$ then $|\G-\G| \gg
|\G|^{2} \log^{-\frac{15}{2}} |\G|$ for an arbitrary multiplicative
subgroup $\G$, $|\G| \ll \sqrt{p}$. Corollary \ref{c:f:sigma(P,S)}
easily implies that $|A-A| \gg |A|^{\frac{3}{2} + \epsilon}$,
$\epsilon >0$ for any convex set or multiplicative subgroup of size
$O( \sqrt{p})$.





\section{Relations between $\E_k(A)$ and $\T_l(A)$}
\label{sec:E_k-T_l}

Notice that from
Corollary \ref{c:E_4} one can deduce that there exists a constant $C>0$ such
that if $\E (A) = |A|^3 / K$ and $\E_4 (A) = M |A|^5 / K^3$ then
$$\T_l(A)\ge \frac{|A|^{2l-1}}{K(CM)^{Cl}}\,.$$
Theorem \ref{t:E_BSzG2} gives similar bound provided by $\E_{3+\eps} (A) = M|A|^{4+\eps} / K^{2+\eps}$.
The proof of Theorem \ref{t:BSzG1} bring up the following question.
Does there exist a set $A$ such that $\E(A)=|A|^3/K,\, \E_3 (A) =
M_1 |A|^4 / K^2$, and $\T_l (A) = M_2 |A|^{2l-1} / K^{l-1}$, $l\ge
3$ with $M_1,M_2$ relatively small simultaneously?
Note that if $\E (A) = |A|^3 / K$ then the estimates $\E_3 (A) \ge
|A|^4 / K^2$, $\T_l (A) \ge |A|^{2l-1} / K^{l-1}$ easily follows
from the Cauchy--Schwarz inequality. Interesting, that the answer is
negative, provided that the assumption on the additive energy we
replace by $|A-A|=K|A|$. It can be deduced from Theorem \ref{t:small
E_3}, but we describe a more direct approach providing a slightly
better lower bounds. Similar arguments were used in
\cite{ss-convex}.

\begin{proposition}
    Let $A \subseteq \Gr$ be a set, and $l\ge 2$ be a positive integer.
    Then
    \begin{equation}\label{f:E_k_and_T_k_difference}
        \left( \frac{|A|^8}{8 \E_3(A)} \right)^l \le \T_l (A) |A-A|^{2l+1} \,,
    \end{equation}
    \begin{equation}\label{p:E_k_and_T_k_sumset}
        \left( \frac{|A|^9}{8  \E_3(A)} \right)^l \le \T_l (A) |A+A|^{3l+1} \,,
    \end{equation}
    and
    \begin{equation}\label{p:E_k_and_T_k_sumset2}
        \left( \frac{|A|^{20}}{32 \E^{3}_3 (A)} \right)^l \le \T_l(A) |A+A|^{6l+1} \,.
    \end{equation}
\label{p:E_k_and_T_k}
\end{proposition}
\begin{proof}
Let $D=A-A$, $S=A+A$, $|D| = K|A|$, and $|S|=L|A|$. As before, we
define $P= \{ s\in D ~:~ |A_s| \ge |A|/(2K) \}$ and $P' = \{ s\in D
~:~ |A_s| \ge |A|/(2L) \}$. Then
$$\sum_{x\in P} |A_s| \ge |A|^2 /2$$
and
\begin{equation*}\label{est:P'}
    \sum_{s\in P'} |A_s|^2 \ge \frac12 \E(A)  \,.
\end{equation*}
By Lemma \ref{l:pop} and the Katz--Koester transform
\begin{equation*}\label{f:our_Katz--Koester}
    \frac{|A|^{6}}{4 \E_3 (A)} \le \sum_{s\in P} |A-A_s| \le \sum_{s\in P} (D\circ D) (s) \,.
\end{equation*}
Thus, by definition of the set $P$, we have
\begin{equation*}\label{f:our_Katz--Koester+}
    \frac{|A|^{7}}{8 K \E_3 (A)} \le \sum_{s} (A \circ A) (s) (D\circ D) (s) \,.
\end{equation*}
Using the Fourier inversion formula and the H\"{o}lder inequality,
we infer that
\begin{equation*}\label{}
    \frac{|A|^{7}}{8 K \E_3 (A)} \le  \int_\xi |\FF{A} (\xi)|^2  |\FF{D}
    (\xi)|^2\le \Big( \int_\xi |\FF{A} (\xi)|^{2l} \Big)^\frac{1}{l} \Big(\int_\xi |\FF{D}
    (\xi)|^{\frac{2l}{(l-1)}}\Big)^{\frac{l-1}{l}}
    \,,
\end{equation*}
so that
\begin{equation*}\label{f:E_3_T_k}
    \left( \frac{|A|^{7}}{8 K \E_3 (A)} \right)^l
        \le
            \T_l (A) \cdot |D|^{l+1}
\end{equation*}
for every  $l\ge 2$ and (\ref{f:E_k_and_T_k_difference}) follows.

Next we prove (\ref{p:E_k_and_T_k_sumset}). For every $1\le  j\le
t:= \lf \log L + 1\rf$ put
$$
    P'_j = \{ s\in D ~:~ 2^{j-1} |A|/(2L) < |A_s| \le 2^{j} |A|/(2L) \} \,.
$$
Further, we have
\begin{equation}\label{f:energy-sum}
    \frac{|A|^3}{2L} \le \sum_{s\in P'} |A_s|^2
        \le
            \frac{|A|}{2L} \sum_{j=1}^t 2^j \sum_{s\in P'_j} |A_s|
               =
                    \frac{|A|}{2L} \sum_{j=1}^t 2^j \delta_j |A|^2 \,,
\end{equation}
where $\delta_j := \frac{1}{|A|^2} \sum_{s\in P'_j} |A_s|$. Whence
\begin{equation*}\label{}
    \sum_{j=1}^t 2^j \delta_j \ge 1
\end{equation*}
and
\begin{equation}\label{tmp:01.06.2011_1}
    \sum_{j=1}^t 2^j \delta^2_j \ge \frac{1}{2L}\,.
\end{equation}
By Lemma \ref{l:pop} applied for arbitrary $j \in [t]$, we obtain
\begin{equation*}\label{}
    \delta^2_j |A|^6 \E^{-1}_3 (A) \le \sum_{s\in P'_j} |A+A_s| \,.
\end{equation*}
By the definition of the sets $P'_j$ and  the Katz--Koester
transform, we get
\begin{equation*}\label{}
    \delta^2_j |A|^6 2^{j-1} \frac{|A|}{2L \E_3 (A)} \le \sum_{s\in P'_j} (A \circ A) (s) (S\circ S) (s) \,,
\end{equation*}
hence by (\ref{tmp:01.06.2011_1})
$$
    \frac{|A|^7}{8 L^2 \E_3 (A)} \le \sum_{s} (A \circ A) (s) (S\circ S) (s) \,.
$$
Again using the Fourier inversion formula one has
$$
   \left( \frac{|A|^7}{8 L^2 \E_3 (A)} \right)^l\le \T_l(A)|S|^{l+1}
$$
and the result follows.

It remains to show (\ref{p:E_k_and_T_k_sumset2}). By the first
estimate from (\ref{f:energy-sum}) and the Cauchy--Schwarz
inequality, we obtain
$$
    \sum_{s\in P'} |A_s| \ge \frac{|A|^6}{4L^2 \E_3 (A)} \,.
$$
Applying Lemma \ref{l:pop}, we get
\begin{equation*}
    \frac{|A|^{14}}{16 L^4 \E^{3}_3 (A)} \le \sum_{s\in P'} |A+A_s| \,.
\end{equation*}
Using the same argument as before, we have
\begin{equation*}\label{}
    \frac{|A|^{15}}{32 L^5 \E^{3}_3 (A)} \le \sum_{s} (A \circ A) (s) (S\circ S) (s) \,,
\end{equation*}
so that
$$
   \left( \frac{|A|^{15}}{32 L^5 \E^{3}_3 (A)} \right)^l \le \T_l(A)|S|^{l+1}
$$
and the
proposition is proved. $\hfill\Box$
\end{proof}

\bigskip

Our next result describes the structure of the sets,   whose energy
$\T_l(A)$ is as small as possible in terms of $|A-A|$. It should be
compared with the main theorem of \cite{bk2}, where a similar
statement  was obtained under weaker assumptions, namely
$\E(A)=|A|^3/K$. Our assumption $|A-A|= K|A|$ is much stronger than
$\E(A)=|A|^3/K,$ but also our description of the structure of $A$ is
much more rigid.

\begin{theorem}\label{t:small-T_4}
 Let $A$ be a subset of an abelian group $\Gr$ such that  $|A-A|=K|A|$ and
$ \T_3(A)\le M|A|^5/K^2.$ Then there exist sets $R\sbeq \Gr$ and
$B\subseteq A$ such that $|R|\ll M^{3/2}|A|/|B|,\, |B|\gg
|A|/KM^{},\, \E(B)\gg |B|^3/M^{9/2}$ and
$$|A\cap (R+B)|\gg |A|/M^{3/2}\,.$$
\end{theorem}
\begin{proof}
Let $P=\{s\in A-A: (A\c A)(s)\ge |A|/3K\}$ and define $S_a$  as in
Theorem \ref{t:E_BSzG2}. By H\"older inequality we have
$$\E(A)=\int_\xi|\h A(\xi)|^4\le \Big ( \int_\xi|\h A(\xi)|^6\Big )^{1/2}\Big(\int_\xi|\h A(\xi)|^2\Big)^{1/2}=
\T_3(A)^{1/2}|A|^{1/2}\le \frac{M^{1/2}|A|^3}{K}\,.$$
Further, by $|A-A|=K|A|,$
$$\frac23 |A|^2\le \sum_{s\in P}(A\c A)(s) \,.$$
Observe that
$$\sum_{s}\sum_{a\in A_s} |S_a\cap S_{a-s}| = \sum_{x\in A} (A\circ P) (x)^2
    \ge \frac{1}{|A|} \big(\sum_{s\in P}(A\c A)(s)\big )^2 \ge \frac49 |A|^3 \,,$$
and
$$\sum_{s\not\in P}\sum_{a\in A_s} |S_a\cap S_{a-s}|\le |A|\sum_{s\not\in
P}|A_s|\le \frac13 |A|^2\,.$$ Hence
$$\sum_{s\in P}\sum_{a\in A_s} |S_a\cap S_{a-s}|\gg|A|^3\,.$$
By
Cauchy--Schwarz inequality and (\ref{e:main-ineq}), we have
\begin{eqnarray*}
    |A|^3 &\ll&
            \left( \sum_{s\in P} \frac{(\sum_{a\in A_s} |S_a\cap S_{a-s}|)^{2}}{\E(A_s,A)} \right)^{1/2}
                    \left( \sum_{s}  \E(A_s,A) \right)^{1/2}\\
    &\ll&
              \Big( \sum_{s\in P} (P\c P)(s)\Big)^{1/2}   \E_3 (A)^{1/2}
             \,,
\end{eqnarray*}
so
\begin{equation}\label{e:e_2-e_4}
\sum_{s\in P} (P\c P)(s) \gg \frac{|A|^{6}}{\E_3(A)}\,.
\end{equation}
On the other hand, we have
$$\Big (\frac{|A|}{K}\Big)^3\sum_{s\in P} (P\c P)(s)\ll \T_3(A)\,,$$
hence
$$\E_3(A)\ge c\frac{|A|^4}{KM^{}}:=\g |A|^4\,,$$
for some constant $c>0.$

Observe that
$$\sum_{ |A_{s}|\le \frac12\g |A|}\E(A,A_{s})\le \sum_{ |A_{s}|\le \frac12\g |A|}|A||A_{s}|^2\le \frac12\E_3(A)\,.$$
Put
$$\b=\max_{|A_{s}|> \frac12\g |A|} \frac{\E(A,A_{s})}{|A||A_{s}|^2}\,.$$
Then, by Lemma \ref{l:E_k-identity}, it follows that
$$\frac12 \E_3(A)\le \sum_{ |A_{s}|> \frac12\g |A|}\E(A,A_{s})\le \beta |A|\sum_{s} |A_{s}|^2=\b |A|\E(A)\,,$$
hence $\b\gg M^{-3/2}.$ Finally, there exists a set $B=A_{s}$ such
that $|B|\gg  |A|/KM$ and
$$\E(A,B)\gg |A||B|^2/M^{3/2}\,.$$
By Cauchy--Schwarz inequality
$$\E(B)\ge  \frac{\E(A,B)^2}{\E(A)}\gg
 \frac{K|B|^4}{M^2|A|}\gg |B|^3/M^{9/2}\,.$$
Notice that
$$\E(A,B)=\sum_{a\in A,\, b\in B}|(a+B)\cap (b+A)|\,,$$
hence for some $r,$ we have
$$|(r+B)\cap A|\ge \frac{\E(A,B)}{|A||B|}\gg |B|/M^{3/2}\,.$$
Moreover,
$$\E(A',B)\ge \E(A,B)-2|(r+B)\cap A||B|^2\,$$
where $A'=A\setminus (r+B).$ Thus, iterating this procedure we
obtain a set $R$ of size $O(M^{3/2}|A|/|B|)$ such that
$$|A\cap (R+B)|\gg |A|/M^{3/2},$$
which completes the proof. $\hfill\Box$

\end{proof}
\bigskip

\begin{remark}
    It is easy to see that the proofs of Theorem \ref{t:E_BSzG2} and Theorem \ref{t:small-T_4}
    relies on the following general inequality
    $$
        \left( \sum_{x\in B} (A\c A) (x) \right)^8
            \le
                |A|^8 \E(B) \E_4 (A) \,.
    $$
    The same argument gives for all $l\ge 1$
    $$
        \left( \sum_{x\in B} (A\c A) (x) \right)^{4l}
            \le
                |A|^{6l-4} \E_l (B) \E_{l+2} (A) \,.
    $$
    Is is  interesting to compare these inequalities for $B=A-A$ with Lemma \ref{l:sigma_k}.
\end{remark}

\bigskip
\noindent We finish the paper with an exposition of  a well-known
result of Katz and Koester \cite{kk}. For a set $G\sbeq A\times B,$ by $A\overset{G}-B$ we mean the set of all
elements $a-b$ such that $(a,b)\in G.$

\begin{theorem}\label{katz-koester}
Suppose that $|A-A|=K|A|$.
Then there is a set $B\sbeq A-A$ or $B\subseteq A$ such
that $|B|\gg |A|/(K^{25/22}\log K)$ and $\E(B)\gg |B|^3/(K^{21/22}\log^{4/11}K).$
\end{theorem}
\begin{proof}
Let $1\le M \le K^{1/22}$ be a real number.
We assume that $\E(B) \le M|B|^3 /K$ for every $B\sbeq A-A$ or $B\subseteq A$ such that $|B|\gg |A|/(K^{25/22}\log K)$.
Our aim is to show that $M$ is large.

Suppose that $\E(A)\le M|A|^3/K.$
Again, let $P\sbeq A-A=D$ be the set of all differences with at least $|A|/2K$ representations. Then
$$\frac12 |A|^2 \le \sum_{s\in P}(A\c A)(s)\le \E(A)^{1/2}|P|^{1/2}\,,$$
hence $|P|\ge \frac14 K|A|/M.$
We consider two cases. First assume that there exists a set $P'\sbeq P$ of size $|P|/2$ such that for every
$s\in P'$ we have
$|A-A_s|\ge K^{1/2}M|A|.$ As in (\ref{f:energy-thm20}) we have
$$\E(D)\ge \sum_{s\in P}(D\c D)(s)^2\ge \sum_{s\in P'}|A-A_s|^2\ge |P'|KM^2|A|^2\gg M|D|^3/K$$
and the assertion follows if we will show that $M$ is large.

Now, assume that  there exists a set $P''\sbeq P$ of size $|P|/2$ such that for every
$s\in P''$ we have
$|A-A_s|< K^{1/2}M|A|.$ Therefore, for each $s\in P'',\, $
 $\E(A,A_s)> |A||A_s|^2/K^{1/2}M,$ so that
$$\sum_x (A\c A)(x)(A_s\c A_s)(x)> \frac{|A||A_s|^2}{K^{1/2}M} \,.$$
Pigeonholing, for each $s\in P''$ there is an $1\le i=i(s)\le \frac12 \log (KM^2)$ such that
\begin{equation}\label{f:pop-A-A}
\sum_{x :\atop |A|/2^i< (A\c A)(x) \le |A|/2^{i-1}} (A_s\c A_s)(x)\gg \frac{2^i|A_s|^2}{K^{1/2}M\log K} \,.
\end{equation}
Thus, there exist $i_0$ and   a set $Q\sbeq P''$ of size $\gg |P|/\log K$ such that for every $s\in Q,\, i(s)=i_0$.
Let $G_s\sbeq A_s^2$ consists of all pairs $(a,a')\in A_s^2$ such that $(A\c A)(a-a')\ge |A|/2^{i_0}.$
By (\ref{f:pop-A-A}) it follows that
$|G_s|\gg 2^{i_0}|A_s|^2/(K^{1/2}M\log K).$ Again we may assume that
$$|A_s\o{G_s}-A_s|\gg \frac{2^{2i_0}|A_s|}{M^3\log K} \,,$$
because otherwise after some choice of constants,
we have $\E(A_s)\ge |G_s|^2/|A_s\o{G_s}-A_s|\ge M|A_s|^3/K.$
Put
$$X=\bigcup_{s\in Q}(A_s\o{G_s}-A_s)$$
and observe that
$|X|{2^{-2i_0}}|A|^2\le \E(A)\le M|A|^3/K,$
so  $|X|\le 2^{2i_0}M|A|/K.$ Define
$$g(x)= |\{s\in Q: x\in A_s\o{G_s}-A_s\}| $$
and notice that if $x\in A_s\o{G_s}-A_s$ then $s\in D\cap(x +D).$
Therefore, assuming $\E(D)\le M|D|^3/K$ and $\E(X)\le M|X|^3/K,$
\begin{eqnarray}\label{e:KK-energybound1}
\frac{|A|^22^{2i_0}}{M^4\log^2 K} &\ll& \sum_{s\in Q}|A_s\o{G_s}-A_s|=\sum_{x\in X}g(x)\le \sum_{x\in X}(D\c D)(x)\nonumber\\
&=&\sum_{d\in D}(D\c X)(d)\le |D|^{1/2}\E(D)^{1/4}\E(X)^{1/4}\nonumber\\
&\le& M^{1/4}K|A|^{5/4}\E(X)^{1/4}\,.
\end{eqnarray}
On the other hand for each $x\in X$  we have  $(A\c A)(x)\ge |A|/2^{i_0}$, so that
\begin{equation}\label{e:KKenergybound2}
2^{-i_0}|A||X|\le \sum_{x\in X}(A\c A)(x)\le M^{1/4}K^{-1/4}|A|^{5/4}\E(X)^{1/4}\,.
\end{equation}
Combining (\ref{e:KK-energybound1}) and (\ref{e:KKenergybound2}), in view of $|A|\ge K|X|/(2^{2i_0}M),$ we see that
 $$\E(X)\gg \frac{|X|^3}{M^{10}K^{1/2}\log^4K}\,.$$
The assertion follows for $M\gg K^{1/22}/\log^{4/11}K.$
$\hfill\Box$
\end{proof}

\bigskip

\no{Faculty of Mathematics and Computer Science,\\ Adam Mickiewicz
University,\\ Umul\-towska 87, 61-614 Pozna\'n, Poland\\} {\tt
schoen@amu.edu.pl}
\bigskip

\no{Division of Algebra and Number Theory,\\ Steklov Mathematical
Institute,\\
ul. Gubkina, 8, Moscow, Russia, 119991\\}
and
\\
Delone Laboratory of Discrete and Computational Geometry,\\
Yaroslavl State University,\\
Sovetskaya str. 14, Yaroslavl, Russia, 150000\\
{\tt ilya.shkredov@gmail.com}


\begin{thebibliography}{99}
\bibitem{abs}
{\sc N.~Alon, B.~Bukh, B.~Sudakov,} {\em Discrete Kakeya-type
problems and small bases,} Israel J. Math. 174 (2009), 285--301.
\bibitem{balog}
{\sc A.~Balog,} {\em Many additive quadruples,} Additive
combinatorics,  CRM Proc. Lecture Notes, 43, Amer. Math. Soc.,
Providence, RI, 2007, 39--49.
\bibitem{Balog_AA+A}
{\sc A.~Balog,} {\em A note on sum--product estimates,} preprint.
\bibitem{balog-szemeredi}
{\sc A.~Balog, E.~Szemer\'edi,} {\em A statistical theorem of set
addition,} Combinatorica, 14 (1994), 263--268.
\bibitem{bk1}
{\sc M.~Bateman, N.~Katz, }
{\em New bounds on caps sets, }
arXiv:1101.5851v1 [math.CA] 31 Jan 2011.
\bibitem{bk2}
{\sc M.~Bateman, N.~Katz, }
{\em Structure in additively nonsmoothing sets, }
arXiv:1104.2862v1 [math.CO] 14 Apr 2011.
\bibitem{BourgainA+A}
{\sc J.~Bourgain, }
{\em On Aritmetic Progressions in Sums of Sets of Integers, }
A Tribute of Paul Erd\"{o}s, Cambridge University Press, Cambridge (1990), 105--109.
\bibitem{cochraine_comment}
{\sc T.~Cochraine, }
{\em  http://gowers.wordpress.com/2011/02/10/a-new-way-of-proving-sumset-estimates/\#comment-10647}.
\bibitem{crs}
{\sc E.~Croot, I.~Z.~Ruzsa, T.~Schoen,} {\em Arithmetic progressions
in sparse sumsets,} Combinatorial number theory, 157--164, de
Gruyter, Berlin, 2007.
\bibitem{cs}
{\sc E.~Croot, O.~Sisask,} {\em A probabilistic technique for
finding almost-periods of convolutions,} Geom. Funct. Anal. 20
(2010), 1367--1396.
\bibitem{gowers}
{\sc W.~T.~Gowers, } {\em A new proof of Szemer\'{e}di's theorem, }
Geom. Funct. Anal. \textbf{11} (2001), 465--588.
\bibitem{Gowers_groups}
{\sc W.~T.~Gowers, } {\em Quasirandom groups,}  Combin. Probab.
Comput. 17 (2008),  363--387.
\bibitem{Haight}
{\sc J.~A.~Haight, } {\em Difference covers which have small $k$--sums for any $k$, }
Mathematika \textbf{20} (1973), 109--118.
\bibitem{ikrt}
{\sc A.~Iosevich, V.~S.~Konyagin, M.~Rudnev, V.~Ten,} {\em On combinatorial complexity of convex sequences,} Discrete Comput. Geom. 35 (2006), 143--158.
\bibitem{Johnsen}
{\sc J.~Johnsen, }
{\em On the distibution of powers in finite fields,} J. Reine Angew. Math., 251, 1971, 10--19.
\bibitem{kk}
{\sc N.~H.~Katz, P.~Koester, } {\em On additive doubling and energy,
} SIAM J. Discrete Math., 24 (2010), 1684--1693.
\bibitem{K_Tula}
{\sc S.~V.~Konyagin, }
{\em Estimates for trigonometric sums and for Gaussian sums}
IV International conference "Modern problems of number theory and its applications". Part 3 (2002), 86--114.
\bibitem{KS1}
\sc{S.~Konyagin, I.~Shparlinski, }
{\em Character sums with exponential functions}
Cambridge University Press, Cambridge, 1999.
\bibitem{Lev+_universal}
{\sc S.~Kopparty, V~.F.~Lev, S.~Saraf, M.~Sudan, } {\em Kakeya--type sets in finite vector spaces, }
arXiv:1003.3736v1 [math.NT].
\bibitem{Li}
{\sc L.~Li, } {\em On a theorem of Schoen and Shkredov on sumsets of convex sets, }
arXiv:1108.4382v1 [math.CO].
\bibitem{p}
{\sc G.~Petridis, } {\em New Proofs of Pl\"{u}nnecke-type Estimates
for Product Sets in Non-Abelian Groups, } preprint.
\bibitem{Rudin_book}
{\sc W.~Rudin, } {\em Fourier analysis on groups,}  Wiley 1990 (reprint of the 1962 original).
\bibitem{ruzsa1}
{\sc I.~Z.~Ruzsa,} {\em Sumsets and structure,} Combinatorial number
theory and additive group theory, 87–210, Adv. Courses Math. CRM
Barcelona, Birkhäuser Verlag, Basel, 2009.
\bibitem{ruzsa2}
{\sc I.~Z.~Ruzsa,} {\em On the cardinality of $A+A$ and $A-A$,}
Combinatorics (Proc. Fifth Hungarian Colloq., Keszthely, 1976), Vol.
II, pp. 933–938, Colloq. Math. Soc. János Bolyai, 18, North-Holland,
Amsterdam-New York, 1978.
\bibitem{sanders1}
{\sc T.~Sanders,} {\em On the
Bogolubov--Ruzsa Lemma,} preprint.
\bibitem{sanders2}
{\sc T.~Sanders,} {\em On Roth's
Theorem on Progressions,} Ann. of Math., to appear.
\bibitem{sanders:pop}
{\sc T.~Sanders,} {\em Popular difference sets,} Online J. Anal.
Comb., 5 (2010), Art. 5, 4 pp.
\bibitem{Sanders_Sh}
{\sc T.~Sanders,} {\em On a theorem of Shkredov,}
availible at arXiv:0807.5100v1 [math.CA] 31 Jul 2008.
\bibitem{ss}
{\sc T.~Schoen, I.~D.~Shkredov,} {\em Additive properties of
multiplicative subgroups of $\F_p$,} to appear in Quart. J. Math.
\bibitem{ss-convex}
{\sc T.~Schoen, I.~D.~Shkredov,} {\em On sumsets of convex sets, }
Comb. Probab. Comput. 20 (2011), 793--798.
\bibitem{s}
{\sc I.~D.~Shkredov, }
{\em Some applications of W. Rudin's inequality to problems
of combinatorial number theory, } UDT, accepted for publication.
\bibitem{shkredov_LES}
{\sc I.~D.~Shkredov, }
{\em On Sets of Large Exponential Sums, }
Izvestiya of Russian Academy of Sciences, 72, N 1, 161--182, 2008.
\bibitem{shkredov_dissociated}
{\sc I.~D.~Shkredov, }
{\em On sumsets of dissociated sets, } Online Journal of Analytic Combinatorics, {\bf 4} (2009), 1--26.
\bibitem{Sh_doubling}
{\sc I.~D.~Shkredov, }
{\em On Sets with Small Doubling, } Mat. Zametki, 84:6 (2008), 927--947.
\bibitem{sv}
{\sc I.~D.~Shkredov, I.~V.~V'ugin,} {\em On additive shifts of
multiplicative subgroups,} Mat. Sbornik, accepted for publication.
\bibitem{sy}
{\sc I.~D.~Shkredov, S.~Yekhanin,}
{\em Sets with large additive energy and symmetric sets,}
Journal of Combinatorial Theory, Series A 118 (2011) 1086--1093.
\bibitem{sol}
{\sc J.~Solymosi, }
{\em An upper bound for the multiplicative energy, }
arXiv:0806.1040v1 [math.CO] 5 Jun 2008
\bibitem{sz-t}
{\sc E.~Szemer\' edi, W.~T.~Trotter,} {\em Extremal problems in
discrete geometry,} Combinatorica 3 (1983), 381–392.
\bibitem{tv}
{\sc T.~Tao, V.~Vu, }{\em Additive combinatorics,} Cambridge University Press 2006.
\end{thebibliography}
\end{document}